\documentclass[12pt,english,reqno]{amsart}

\usepackage{amsfonts,amsmath,latexsym,verbatim,amscd,mathrsfs,color,array}
\usepackage{amssymb}
\usepackage{pdfsync}
\usepackage{epstopdf}
\usepackage[english]{babel}
\usepackage[latin1]{inputenc} %%% este comando sierve para que las tildes se pongan como es usual%%%
\usepackage{bm}
\usepackage{upgreek}
\usepackage{amsmath}
\usepackage{amssymb}
\usepackage{amsfonts}
\usepackage{dsfont}
\usepackage{amsthm}
\usepackage{latexsym}
\numberwithin{equation}{section}
\usepackage{color}      % Need the color package
\newcommand{\R}{\mathbb{R}}

\newcommand{\LL}{\mathcal{L}}

\newcommand{\pp}{\partial}
\newcommand{\e}{\textrm{e}}
\newcommand{\y}{\textrm{y}}

\newcommand{\D}{\delta}
\newcommand{\U}{\mathcal{U}}

\newcommand{\la}{\lambda}

\newcommand{\ep}{\varepsilon}

\definecolor{crimson}{rgb}{0.86,0.05,0.24}

\definecolor{airforceblue}{rgb}{0,0.15,0.8}

\newcommand{\rdiv}{{\rm div}}
\newcommand{\rdist}{{\rm dist}}

\newtheorem{lemma}{Lemma}[section]
\newtheorem{prop}{Proposition}[section]
\newtheorem{theo}{Theorem}[section]

\title{an existence result for anisotropic quasilinear problems}

\author[O. Agudelo]{Oscar Agudelo}
\address{\noindent O. Agudelo - NTIS, Department of Mathematics,
Z\'{a}pado\v{c}esk\'{a} Univerzita v Plzni, Plze\v{n}, Czech Republic.}
\email{oiagudel@ntis.zcu.cz}

\author[P. Dr\'{a}bek]{Pavel Dr\'{a}bek}
\address{\noindent P. Dr\'{a}bek - ZCU, Department of Mathematics,
Z\'{a}pado\v{c}esk\'{a} Univerzita v Plzni, Plze\v{n}, Czech Republic.}
\email{pdrabek@kma.zcu.cz}

\begin{document}
\maketitle

%\tableofcontents

\begin{abstract}
We study existence of solutions for a boundary degenerate (or singular) quasilinear equation in a smooth bounded domain under Dirichlet boundary conditions. We consider a weighted $p-${L}aplacian operator with a coefficient that is {locally bounded  inside the domain and satisfying certain additional integrability assumptions}. Our main result applies for boundary value problems involving continuous non-linearities having no growth restriction, but provided the existence of a sub and a supersolution is guaranteed. As an application, we present an existence result for a boundary value pro\-blem with a non-linearity $f(u)$ satisfying $f(0) \leq 0$ and having $(p-1)-$sublinear growth at infinity.
\end{abstract}

\medskip


\medskip


\section{Introduction}\label{introduction}
In this work we study the degenerate (or singular) boundary value problem ({\it BVP})
\begin{equation}\label{MODELDEGQUASILINEAR}
\left\{
\begin{aligned}
-\rdiv\left(a(x)|\nabla u|^{p-2}\nabla u\right)& =  F(x,u) \quad & \hbox{in} & \quad \,\,\Omega\\
u&=0 \quad & \hbox{on} & \quad  \pp \Omega,
\end{aligned}
\right.
\end{equation}
where {${N} \geq 2$}, {$p\in (1,N)$} and $\Omega\subset \R^N$ is a smooth bounded domain.

%$F=F(x,\zeta)$ is measurable in $x\in \Omega$ for every $\zeta\in \R$ and is continuous in $\zeta$ for a.e. $x\in \Omega$.

\medskip
In order to formulate our assumptions on the function $a:\Omega \to \R$, we  consider first $\rho_0>0$ small so that in the set $
\Omega_{\rho_0}:=\{x\in \Omega\,:\, 0<\rdist(x,\pp \Omega)< \rho_0\}$
the distance function 
\begin{equation}\label{def:distfunc}
\Omega_{\rho_0}\ni x \mapsto{\rm dist}(x, \partial \Omega)
\end{equation}
{satisfies that given any $x\in \Omega_{\rho_0}$, there exists a unique $z_x \in \partial \Omega$  such that ${\rm dist}(x,\Omega)={\rm dist}(x,z_x)$ (see also subsection \ref{2.3}).} We consider also a {\it positive function} ${\sf a}:(0,\rho_0) \to (0,\infty)$ and a fixed positive number $s$ such that 
\begin{equation}\label{hyp:rma}
{\sf a}\in L^1(0,\rho_0)\cap L^{\infty}_{loc}(0,\rho_0), \quad {\sf a}^{-1}\in L^s(0,\rho_0)\quad  \hbox{and} \quad s>\max \left(\frac{N}{p}, \frac{1}{p-1}\right).
\end{equation}

{I}n particular,  ${\sf a}^{-\frac{1}{p-1}}\in L^{1}(0,\rho_0)$.

\medskip
The assumptions for the function $a$ are the following: 

\medskip
{\it (A1)} $a\in L^{\infty}_{loc}(\Omega)$ and for any open set $D$ with $\overline{D}\subset \Omega$, $\inf \limits_{{	D}}a>0$;

\medskip
{\it (A2)} for every $x \in \Omega_{\rho_0}$, 
\begin{equation*}a(x)={\sf a}\big(\rdist(x,\pp \Omega)\big).
\end{equation*}

{Assumptions {\it (A1)} and {\it (A2)} imply that the differential ope\-rator in \eqref{MODELDEGQUASILINEAR} is uniformly elliptic on any subdomain of $\Omega$, but not necessarily near the boundary of the domain.}

\medskip
{These hypotheses also yield that $a\in L^1(\Omega)$ and $a^{-1}\in L^s(\Omega)$ (see Section \ref{2.1}), which will allow us to put the equation \eqref{MODELDEGQUASILINEAR} {into} an appropriate functional analytic setting.

\medskip
To formulate the assumptions on the function $F:\Omega \times \R \to \R$, we first let $s>0$ be as in \eqref{hyp:rma} and denote 
\begin{equation}\label{sobolevexponents}
p_s\,:=\,\frac{p\,s}{s+1}, 
\quad \quad p_s^*:=\,\frac{N p_s}{N-p_s}\
=\frac{Nps}{N(s+1) - ps}.
\end{equation}

Using that $s> \frac{N}{p}>1$, we find that
\begin{equation}\label{exponentsinequalities}
p_s\,<\,p\,<\,\,p_s^*\,\,<\,\,\frac{Np}{N-p}.
\end{equation}

Next, fix $q\in [p,p_s^*)$ {and set} $\frac{q}{q-p}:=\infty$, if $q=p$. Consider also a measurable function $b$ satisfying the following:

\medskip
{\it (B1)} ${b\in L^{\frac{q}{q-p}}(\Omega)}$;

\medskip
{\it (B2)} there exist constants $0<c_1<c_2$ such that 
\begin{equation*}\label{behaviorb(x)}
c_1 \leq \liminf \limits_{{\rm dist}(x, \pp \Omega) \to 0}\,a^{\frac{1}{p-1}}(x)b(x) \leq  \limsup \limits_{{\rm dist}(x, \pp \Omega) \to 0}\,a^{\frac{1}{p-1}}(x)b(x) \leq c_2.
\end{equation*}

In particular, $b \neq 0$ and $b \geq 0$ a.e. in $\Omega$.

\medskip

\medskip
As for the function $F:\Omega\times \R \to \R$, we assume it is a Caratheodory function satisfying the growth condition:

\medskip
for every $M>0$, there exists $C_M>0$ such that for every $\zeta\in\R$ with $|\zeta|\leq M$ and for a.e. $x\in \Omega$,
\begin{equation}\label{growthF(x,u)0}
|F(x,\zeta)|\leq C_M b(x).
\end{equation}

\medskip

We use {\it {w}eighted Sobolev spaces} to study the BVP \eqref{MODELDEGQUASILINEAR}. The Sobolev space $W^{1,p}(\Omega,a)$ is defined as the class of functions $v\in W^{1,p}_{loc}(\Omega)$ such that
$$
\|v\|_{W^{1,p}(\Omega,a)}:= \left(\int_{\Omega}|v|^p dx + \int_{\Omega}a(x)|\nabla v|^pdx\right)^{\frac{1}{p}}<\infty.
$$

The Sobolev space $W^{1,p}_0(\Omega,a)$ is defined as the closure of $C^{\infty}_c(\Omega)$ with respect to the norm $\|\cdot\|_{W^{1,p}(\Omega,a)}$.

\medskip
A function $u\in W^{1,p}_0(\Omega,a)$ is called a {\it solution} of \eqref{MODELDEGQUASILINEAR} if $F(\cdot,u(\cdot))\in L^{\frac{p}{p-1}}(\Omega)$ and
\begin{equation}\label{weakaslnMODELDEGQUASILINEAR}
{\forall \varphi \in W^{1,p}_0(\Omega,a)}: \quad \int_{\Omega} a(x)|\nabla u|^{p-2}\nabla u \cdot \nabla \varphi dx = \int_{\Omega}F(x,u) \varphi dx.
\end{equation}

\medskip

A function $u\in W^{1,p}(\Omega,a)\cap C(\overline{\Omega})$ is called a {\it subsolution (supersolution  respectively)} of \eqref{MODELDEGQUASILINEAR} if $F(\cdot,u(\cdot))\in L^{\frac{p}{p-1}}(\Omega)$, for every $\varphi \in W^{1,p}_0(\Omega,a)$ with $\varphi \geq 0$ a.e. in $\Omega$,
\begin{equation}\label{weakaslnMODELDEGQUASILINEAR2}
\int_{\Omega} a(x)|\nabla u|^{p-2}\nabla u \cdot \nabla \varphi dx  \leq (\geq)\int_{\Omega} F(x,u) \varphi  dx
\end{equation}
and $\max \limits_{x\in \pp \Omega}u(x) \leq 0\,\,$ $\left(\min \limits_{x\in\pp \Omega}u(x) \geq0\right)$.

\medskip

Our first main result reads as follows.

 \begin{theo}\label{multiplicitytheorem}
Let {\it (A1)}-{\it (A2)} and {\it (B1)}-{\it (B2)} hold and let $F$ satisfy \eqref{growthF(x,u)0}. Assume that $\underline{u}$ is a subsolution and $\overline{U}$ is a supersolution to \eqref{MODELDEGQUASILINEAR} such that $\underline{u} \leq \overline{U}$ in $\Omega$. Then the BVP \eqref{MODELDEGQUASILINEAR} has a solution $u \in C(\overline{\Omega})$ such that $\underline{u}\leq u \leq \overline{U}$ in $\Omega$.
\end{theo} 

\medskip

The proof of Theorem \ref{multiplicitytheorem} is motivated by the techniques described in Chapter 5 in \cite{DRABEKMILOTA}. As an application of Theorem \ref{multiplicitytheorem}, we present existence of a positive solution for the BVP 
\begin{equation}\label{DEGQUASILINEAR}
\left\{
\begin{aligned}
-\rdiv\left(a(x)|\nabla u|^{p-2}\nabla u\right) &= \la \,b(x)\,f(u) \quad &\hbox{in}& \quad \,\,\Omega\\
u&=0 \quad  &\hbox{on}& \quad  \pp \Omega,
\end{aligned}
\right.
\end{equation}
where $\la>0$  is a positive parameter and $f:[0,\infty) \to\R$ is a continuous function satisfying the following hypotheses:

\medskip
{\it (f1)} $f$ is $(p-1)-$sublinear, i.e., $\limsup \limits_{\zeta\to \infty}\frac{f(\zeta)}{\zeta^{p-1}}= 0$; 

{\it (f2)} there exists $r\in (0,p-1)$ such that $\liminf \limits_{\zeta\to \infty}\frac{f(\zeta)}{\zeta^{r}}\in (0,\infty]$.

\medskip
{\it (f3)} $f(0)\leq 0$ and $\min \limits_{0\leq \zeta< \infty}f(\zeta)<0$.

\medskip

Problem \eqref{DEGQUASILINEAR} finds its applications, for instance, in resource management models, see \cite{CASTROMAYASHIVAJI}, and design of suspension bridges, see \cite{GAZZOLA}.

\medskip
{We present the function 
$$
f(\zeta) = -A + \frac{\zeta^{p-1}}{\log(2+ \zeta)} \quad \hbox{for} \quad \zeta \geq 0,
$$
where $A\geq 0$, as an example of a nonlinearity satisfying conditions {\it (f1)}, {\it (f2)} and {\it (f3)}.} 

\medskip
Condition $f(0)\leq 0$ in {\it (f3)} implies that $u=0$ is a supersolution to \eqref{DEGQUASILINEAR}. Under this condition, finding a positive subsolution for \eqref{DEGQUASILINEAR} becomes a subtle matter. We refer the reader to \cite{ALICASTROSHIVAJI,AMBROSSARCOYABUFFONI, BRWONCASTROSHIVAJI,HAISHIVAJI, LEESHIVAJIYE} and references therein, for results concerned with {related boundary value}  problems in the uniformly elliptic case.

\medskip
The authors in \cite{AGUDELODRABEK1} showed the existence of a positive solution of  equation \eqref{DEGQUASILINEAR} either in a ball or in $\R^N$, under a similar set of assumptions on $a, b$ and $f$, but assuming in addition the global radial symmetry of the coefficient $a$ and instead of {\it (f3)}, assuming that there exist $K_0, \delta >0$ and $\gamma \in (0,1)$ such that
{$$
\hbox{for every} \quad \zeta \in (0,\delta), \qquad -\frac{K_0}{\zeta^{\gamma}} \leq f( \zeta) \leq 0.
$$}

\medskip
As remarked in \cite{AGUDELODRABEK1}, our hypotheses do not allow, in general, to obtain either $C^{1}_{loc}-$regularity of weak solutions to \eqref{DEGQUASILINEAR}, see \cite{TOLSKDORF1}, or Hopf's boundary lemma, see \cite{VAZQUEZJL}. Instead, regularity of solutions relies on the local H\"{o}lder continuity due to Serrin, see \cite{SERRIN}. In particular, {\it global a priori $C^1$-estimates} and therefore the construction of a subsolution of \eqref{DEGQUASILINEAR} is a much more challenging task. In {this work}, tools from differential geometry help us overcome this difficulty.

\medskip
{More precisely}, we say that $\pp \Omega$ has {\it non-negative mean curvature} if for every $x\in \pp \Omega$, 
\begin{equation}\label{positivemeancurvature0}
\sum_{i=1}^{N-1} k_{i}(x)\geq 0,
\end{equation}
where $k_{1}, \ldots,k_{N-1}$ are 
the principal curvatures 
of $\pp \Omega$.

\medskip
This geometric notion generalizes convexity and it is  equivalent to the monotonicity of the surface area element of 
$\pp \Omega$, see \cite{GROMOV}. This monotonicity states that the $(N-1)$-dimensional 
volume of every subdomain 
$\omega \subset \pp \Omega$
does not decrease as $\omega$ approaches $\pp \Omega$.

\medskip
From the {\it PDE} point of view, condition \eqref{positivemeancurvature0} implies that the distance function $\rdist(\cdot,\pp \Omega)$ is subharmonic in any subdomain of $\Omega_{\rho_0}$ (see \eqref{def:distfunc}) and in essence this fact will allow us to construct an appropriate subsolution to \eqref{DEGQUASILINEAR}.

\medskip
Our second main result reads as follows.

\begin{theo}\label{theo2}
Let $\Omega$ be a smooth bounded domain such that $\pp \Omega$ has non-negative mean curvature. Let (A1)-(A2) and (B1)-(B2) hold and let $f\in C[0,\infty)$ satisfy {(f1)-(f3)}. Then, there exists $\la_0>0$ such that for every $\la \geq \la_0$ the BVP \eqref{DEGQUASILINEAR} has a solution $u\in C(\overline{\Omega})$ which is positive in $\Omega$.
\end{theo}

We refer the reader to \cite{DRABEKLAKSHMI,KAUFMANNRAMOS2018} for similar results in the uniformly elliptic setting.

\medskip
In the case $a(x)\equiv 1$ and $p=2$, under slightly more restrictive assumptions on $f$ and with no assumptions on the geometry of $\pp \Omega$, the authors in \cite{DANCERSHI2006} proved uniqueness of positive solutions of   \eqref{DEGQUASILINEAR} for $\lambda>0$ large. We also refer the reader to \cite{HERRONLOPERA2015} for non-existence of positive radial solutions of a related BVP in the superlinear setting.

\medskip
{We have considered assumption {\it (A2)} for the sake of clarity in our presentation. Although, we remark that Theorems \ref{multiplicitytheorem} and \ref{theo2} are still valid if instead of {\it (A2)}, we assume the existence of functions ${\sf a}_1, {\sf a}_2: (0,\rho_0) \to (0,\infty)$, both satisfying \eqref{hyp:rma} and such that
$$
{\sf a}_1\big({\rm dist}(x,\partial \Omega)\big) \leq a(x) \leq {\sf a}_2\big({\rm dist}(x,\partial \Omega)\big)
$$  
for every $x\in \Omega_{\rho_0}$. The proofs require only straight forward changes and details are left to the reader.} 

\medskip
The paper is organized as follows. In Section 2 we present the preliminary results and a priori estimates needed for the proof of Theorem \ref{multiplicitytheorem}.  The proof of Theorem \ref{multiplicitytheorem} is presented in Section 3. In Section 4 we provide the proof Theorem \ref{theo2}.

%%%%%%%%%%%%%%%%%%%%%%%%%%%%%%%%%%%%%%%%
%%%%%%%%%%%%%%%%%%%%%%%%%%%%%%%%%%%%%%%%
%          Framework and Preliminar                  
%                 Discussion
%%%%%%%%%%%%%%%%%%%%%%%%%%%%%%%%%%%%%%%%
%%%%%%%%%%%%%%%%%%%%%%%%%%%%%%%%%%%%%%%%

\section{Preliminaries}\label{preliminaries}

\subsection{Weighted Sobolev spaces}\label{2.1}
Let $p\in (1,N)$, $\Omega\subset \R^N$ be a smooth bounded domain and let $a:\Omega \to \R$ satisfy {\it (A1)} and {\it (A2)}.

\medskip
From these hypotheses and 
with the help of {\it Fermi coordinates}, we verify in subsection \ref{2.3} that $a\in L^1(\Omega)$ and $a^{-1}\in L^s(\Omega)$. Observe also that the H\"{o}lder inequality and the fact that $s>\frac{1}{p-1}$ yield $a^{-\frac{1}{p-1}}\in L^{1}(\Omega)$.

\medskip
Thus, the spaces $W^{1,p}(\Omega,a)$ and $W_0^{1,p}(\Omega,a)$, defined in Section 1, are separable. They are also uniformly convex and therefore {\it reflexive Banach spaces} (see Chapter 1 in \cite{DRABEKKUFNERNICOLOSI}).

\medskip

We also notice from {\it (A1)} that for any $D$ with  $\overline{D} \subset \Omega$,
$$
W^{1,p}(D,a)=W^{1,p}(D) \quad  \hbox{and} \quad  W_0^{1,p}(D,a)=W_0^{1,p}(D).
$$

In what follows $\mathds{1}_{E}$ denotes the characteristic function of a set $E\subset \R^N$. Also, for any $t\in \R$, write
\begin{equation}\label{defposnegparts}
\begin{aligned}
t^+:= &\frac{t+ |t|}{2},
\qquad  t^-:=\frac{t-|t|}{2},
\end{aligned}
\end{equation}
so that $t^+ = \max(t,0)$ and $t^-= \min(t,0)$. 

\medskip
Our first lemma concerns with algebraic properties of  the spaces $W^{1,p}(\Omega,a)$ and $W_0^{1,p}(\Omega,a)$.

\begin{lemma}\label{positivenegativeparts}
Under the hypotheses {\it (A1)} and {\it (A2)},
\begin{itemize}
\item[(i)] if $v\in W^{1,p}(\Omega,a)$, then $|v|, v^+,v^-$ belong to $W^{1,p}(\Omega,a)$ and
\begin{equation}\label{gradientpositnegatparts}
\nabla v^+= \mathds{1}_{\{v>0\}}\nabla v, \qquad  \nabla v^-= \mathds{1}_{\{v<0\}}\nabla v;
\end{equation}

\medskip

\item[(ii)] 
if $v\in W^{1,p}_0(\Omega,a)$, then $|v|,v^+,v^- \in W^{1,p}_0(\Omega,a)$ and

\medskip
\item[(iii)] if $v,w\in W^{1,p}(\Omega,a)$, then $\min(v,w), \max(v,w)\in W^{1,p}(\Omega,a)$.
\end{itemize}
\end{lemma}

\begin{proof} {\rm (i)} From {\it (A1)} and Theorem 7.8 in \cite{GILBARGTRUDINGER}, $|v|,v^+,v^- \in W^{1,p}_{loc}(\Omega)$ and \eqref{gradientpositnegatparts} holds. Directly from \eqref{defposnegparts} and  \eqref{gradientpositnegatparts}, it follows that $|v|,v^+ , v^- \in W^{1,p}(\Omega,a)$.

\medskip

{\rm (ii)} Let $v\in W^{1,p}_0(\Omega,a)$. There exists $\{v_n\}_n\subset C^{\infty}_c(\Omega)$ such that $v_n\to v$ strongly in $W^{1,p}_0(\Omega,a)$. Let $\{\ep_n\}_n \subset (0,\infty)$ be such that $\ep_n \to 0$ as $n\to \infty$ and set
$$
w_n:= \sqrt{v_n^2 + \ep_n^2} -\ep_n \quad \hbox{for }\quad n\in \mathbb{N}.
$$

Observe that $w_n \in C^{\infty}_{c}(\Omega)$ since ${\rm supp}\,w_n = {\rm supp}\,v_n$. Also, it is direct to verify that $w_n \to |v|$ strongly in $W^{1,p}(\Omega,a)$ as $n\to \infty$. Therefore, $|v|\in W^{1,p}_0(\Omega,a)$ and from \eqref{defposnegparts}, $v^+,v^-\in W^{1,p}_0(\Omega,a)$.

\medskip

{\rm (iii)} Let $v,w\in W^{1,p}(\Omega,a)$. The proof is a direct consequence of {\it (i)} and the fact that 
$$
\min (v,w)= (v-w)^- + w \quad \hbox{and} \quad\max (v,w) = (w-v)^+ + v.
$$ 
\end{proof}

\begin{lemma}\label{approxbyC1uptobdry}
Assume {\it (A1)} and {\it (A2)}. Then, $W^{1,p}_0(\Omega,a)$ is the closure of the  subspace 
$$
X:=W^{1,p}(\Omega,a)\cap \{v\in C(\overline{\Omega})\,:\, v|_{\pp \Omega}=0\}
$$
with respect to the norm $\|\cdot\|_{W^{1,p}(\Omega,a)}$.
\end{lemma}

\begin{proof} {Observe that $C_c^{\infty}(\Omega)\subset X$ and hence $W^{1,p}_0(\Omega,a)$ is contained in the closure of $X$. To prove the reverse inclusion, we claim first that $X$ is the closure of $W^{1,p}(\Omega,a)\cap C_c(\Omega)$ with respect to the norm $\|\cdot\|_{W^{1,p}(\Omega,a)}$.}

\medskip
To prove the claim we follow the lines of the proof of Theorem 2 from section 5.5 in \cite{EVANSLAWRENCEBOOK}. This is done in two steps.

\medskip
{\it Step 1.} Let $u\in X$ and let $x_* \in \pp\Omega$ be arbitrary, but fixed. With no loss of generality, assume that $x_*=(x'_*,0)$ with $x_*'\in \R^{N-1}$. For any $\rho\in (0,\rho_0)$ (see \eqref{def:distfunc}), denote 
$$
\mathcal{C}_{\rho}(x_*):=D^{N-1}_{\rho}(x_*') \times (-\rho,\rho),
$$
where $D_{\rho}^{N-1}(x_*')$ is the ball in $\R^{N-1}$ of radius $\rho$ and centered at $x_*'$.

\medskip
After a local flattening procedure and relabelling of the axis, we may also assume that there exists $\rho_* \in (0,\rho_0)$ such that
$$
\begin{aligned}
\Gamma &:=\pp \Omega \cap \mathcal{C}_{\rho_*}(x_*)= \{(x',x_N)\,:\, x_N=0\}\cap \mathcal{C}_{\rho_*}(x_*) ,\\
\Omega_* &:= \Omega \cap \mathcal{C}_{\rho_*}(x_*)= \{(x',x_N)\,:\, x_N > 0\}\cap C_{\rho_*}(x_*).
\end{aligned}
$$

Let $\ep>0$ be arbitrary, but fixed. Since $u\in W^{1,p}(\Omega,a)$, there exists $m_*\in \Bbb N$ with $m_* \geq \frac{1}{2\rho_*}$ such that 
\begin{equation}\label{intgradrho*}
\forall\, m \geq m_*:\,\,\int_{\Gamma}\int_0^{\frac{2}{m}} a(x)|\nabla u|^pdx <\ep.
\end{equation}

From \eqref{hyp:rma} we may also  assume that for any $m \geq m_*$,
\begin{equation}\label{integrabilitya}\left(\int_{\frac{1}{m}}^{\frac{2}{m}}{\sf a}(\zeta)d\zeta\right)\left(\int_{0}^{\frac{2}{m}}{\sf a}^{-\frac{1}{p-1}}(\zeta)d \zeta\right)^{p-1}  < \ep.
\end{equation}

 Let $\chi:\R \to [0,1]$ be a smooth cut-off function with $\chi \equiv 0$ in $(-\infty,1]$ and $\chi \equiv 1$ in $[2,\infty)$.

\medskip

For any $m\in \mathbb{N}$ with $m \geq m_*$, set 
$$
\begin{aligned}
\chi_{m}(x_N):=&\chi(m x_N) \quad &\hbox{for}& \quad 0<x_N< \infty,\\
u_{m}(x):=& \chi_{m}(x_N)\,u(x) \quad &\hbox{for}& \quad x\in \overline{\Omega}_*. 
\end{aligned}
$$

Observe that $u_m\in W^{1,p}(\Omega_*,a)\cap C(\overline{\Omega}_*)$, $u_m = 0$ in $\Gamma \times (0,\frac{1}{m})$ and
$$
\nabla u_m = \chi(m x_N)\nabla u + m \chi'(m x_N)u\quad \hbox{a.e. in}\quad \Omega_*.
$$

Thus, there exists a constant $C>0$, depending only on $p$, such that {for any $m\geq m*$ and a.e. in $\Omega_*$,}
$$
|\nabla (u_m - u)|^p \leq C(1- \chi(m x_N))^p|\nabla u|^p + C m^p |\chi'(m x_N)|^p|u|^p.
$$

Denote $V_*:= \Omega \cap \mathcal{{C}}_{\frac{\rho_*}{2}}(x_*)$. Taking $C$ larger if necessary, but still depending only on $p$,
\begin{equation}\label{estimateIm+IIm}
\begin{aligned}
\int_{V_*}a(x)|\nabla (u_m - u)|^pdx &\leq \underbrace{C \int_{\Gamma}\int_{0}^{\frac{2}{m}}a(x)|\nabla u|^p dx}_{I_m} \\
& \hspace{3cm}+ \underbrace{Cm^p\int_{\Gamma}\int_{\frac{1}{m}}^{\frac{2}{m}}a(x',x_N)|u(x',x_N)|^pdx_N dx'}_{II_m}.  
\end{aligned}
\end{equation}

From \eqref{intgradrho*}, for every $m \geq m_*$, 
\begin{equation}\label{estimateIm}
0\leq I_m < C \ep.
\end{equation}

Next, we estimate $II_m$. Let $(x',x_N)\in \Gamma\times (\frac{1}{m},\frac{2}{m})$. Using  Theorem 2 from section 4.9.2 in \cite{EVANSGARIEPY1992}, we estimate,
\begin{equation}\label{estimateu}
\begin{aligned}
|u(x',x_N)|&\leq \underbrace{|u(x',0)|} \limits_{=0} + \int_0^{x_N}x_N|\nabla u(x',\zeta)|
d\zeta\\
&\leq x_N \left(\int_{0}^{x_N}{a}^{-\frac{1}{p-1}}(x',\zeta)d\zeta\right)^{\frac{p-1}{p}} \left(\int_{0}^{x_N}a(x',\zeta)|\nabla u(x',\zeta)|^p d\zeta\right)^{\frac{1}{p}}.
\end{aligned}
\end{equation}

From {\it (A2)} and \eqref{estimateu}, 
$$
\begin{aligned}
II_m &\leq C m^p\,\int_{\Gamma}\int_{\frac{1}{m}}^{\frac{2}{m}}{\sf a}(x_N) x_N^p \left(\int_{0}^{x_N}{\sf a}^{-\frac{1}{p-1}}(\zeta)d\zeta\right)^{\frac{1}{p-1}} \left(\int_{0}^{x_N}a(x',\zeta)|\nabla u(x',\zeta)|^p d\zeta\right)dx_N dx'\\
&\leq C \left(\int_{\frac{1}{m}}^{\frac{2}{m}}{\sf a}(x_N)dx_N \right)\left(\int_{0}^{\frac{2}{m}}{\sf a}^{-\frac{1}{p-1}}(\zeta)d\zeta\right)^{\frac{1}{p-1}}\int_{\Gamma} \int_{0}^{\frac{2}{m}}a(x)|\nabla u|^p dx,
\end{aligned}
$$
where $C>0$ is taken again larger if necessary, but still depending only on $p$.

\medskip
Putting together \eqref{intgradrho*}, \eqref{integrabilitya} {and the latter inequalities},
\begin{equation}\label{estimateIIm}
\begin{aligned}
II_m & \leq  C \left(\int_{\frac{1}{m}}^{\frac{2}{m}}{\sf a}(x_N)dx_N \right)\left(\int_{0}^{\frac{2}{m}}{\sf a}^{-\frac{1}{p-1}}(\zeta)d\zeta\right)^{\frac{1}{p-1}} I_m\\
& \leq C\ep^2.
\end{aligned}
\end{equation}

Expressions \eqref{estimateIm+IIm}, \eqref{estimateIm} and \eqref{estimateIIm} yield
$$
\forall\, m\geq  m_*\,:\,\,\int_{V_*}a(x)|\nabla (u_m -u)|^p dx\leq  C(1 + \ep)\ep.
$$

Since $u_m \to u$ in $L^{\infty}(V_*)$, we assume also that 
\begin{equation*}\label{Lpnorm}
\forall\, m\geq  m_*\,:\,\,\|u_m-u\|_{L^p(V_*)} \leq \ep
\end{equation*}
and therefore, for some $\tilde{C}>0$ independent of $\ep$, and for every $m \geq m_*$,
$$
\|u_m -u\|_{W^{1,p}(V_*,a)}\leq \tilde{C}\ep.
$$

\medskip
{\it Step 2.} {Let $v\in X$ and consider $\{V_i\}_{i=0}^n$} an open covering for $\overline{\Omega}$, where for $i=1,\ldots,m$, $V_i$ is a neigborhood of a point $x_i \in \pp \Omega$ as described in Step 1 and $\overline{V}_0\subset \Omega$.

\medskip
{Let $\ep>0$ be arbitrary, but fixed. From Step 1,} there exists $\tilde{\rho}\in (0,\rho_0)$ and for any $i=1,\ldots,m$, there exists $v_i\in W^{1,p}(V_i,a)\cap C(\overline{V_i})$ such that 
$$
\|v_i- v\|_{W^{1,p}(V_i,a)} \leq \ep
$$ 
and $v_i=0$ in  $V_i \cap \Omega_{\tilde{\rho}}$, where 
$$
\Omega_{\tilde{\rho}}:= \{x\in \Omega\,:\,0< {\rm dist}(x,\pp \Omega)<\tilde{\rho}\}.
$$

Since $W^{1,p}(V_0,a)=W^{1,p}(V_0)$, Theorem 1 from Section 5.3.1 in \cite{EVANSLAWRENCEBOOK} yields the existence of $v_0\in C^{\infty}(\overline{V}_0)$ such that 
$$
\|v_0 - v\|_{W^{1,p}(V_0,a)} \leq \ep.
$$  

Let $\{\chi_i\}_{i=0}^m$ be a smooth partition of the unity associated to the covering $\{V_i\}_{i=0}^m$. Set ${\rm v}_i:=\chi_i v_i$ and $v^*:= \sum_{i=0}^m {\rm v}_i$. Observe that $v^* \in W^{1,p}(\Omega,a)\cap C_c(\overline{\Omega})$.

\medskip
We estimate
$$
\begin{aligned}
\|v^* -v\|_{W^{1,p}(\Omega,a)} \leq &\sum_{i=0}^m \|{\rm v}_i - \chi_i v\|_{W^{1,p}(V_i,a)}\\
\leq & \hat{C} \sum_{i=0}^m \|v_i - v\|_{W^{1,p}(V_i,a)}\\
\leq & \hat{C}\,(m+1)\,\ep
\end{aligned}
$$
for some constant $\hat{C}>0$ depending only on $p$ and $\Omega$.

\medskip
Since {$\ep>0$ and $v\in X$ are arbitrary}, we conclude that $X$ is the closure of $W^{1,p}(\Omega,a)\cap C_c(\Omega)$ with respect to the norm $\|\cdot\|_{W^{1,p}(\Omega,a)}$ which proves the claim.

\medskip

We finish the proof now. Let $u\in 
X$ and let $\ep>0$. From Steps 1 and 2, there exists $v\in W^{1,p}(\Omega,a)\cap C_c({\Omega})$ such that 
$$
\|u-v\|_{W^{1,p}(\Omega,a)}\leq \ep.
$$

Let $D$ be an open set such that ${\rm supp}v \subset D \subset \overline{D}\subset \Omega$. From {\it (A1)}, $W^{1,p}_0(D,a)=W^{1,p}_0(D)$ and for every $w\in W^{1,p}_0(D,a)$,
$$
\left(\inf \limits_{D} a\right)^\frac{1}{p}\|w\|_{W^{1,p}_0(D)} \leq \|w\|_{W^{1,p}_0(D,a)}\leq \|a\|_{L^{\infty}(D)}^\frac{1}{p}\|w\|_{W^{1,p}_0(D)}.
$$

On the other hand, since $v\in  W^{1,p}_0(D)$, there exists $w\in C^{\infty}_c(D)$ such that 
$$
\|v-w\|_{W^{1,p}_0(D)}\leq \|a\|^{-\frac{1}{p}}_{L^{\infty}(D)}\ep.
$$

Therefore, 
$$
\begin{aligned}
\|u-w\|_{W^{1,p}_0(\Omega,a)}  &\leq \|u-v\|_{W^{1,p}_0(\Omega,a)} + \|v-w\|_{W^{1,p}_0(\Omega,a)} \\
& \leq \ep+ \|v-w\|_{W^{1,p}_0(D,a)}\\
& \leq \ep+ \|a\|_{L^{\infty}(D)}^{\frac{1}{p}}\|v-w\|_{W^{1,p}_0(D)}\\
& \leq 2\ep.
\end{aligned}
$$

Since {$\ep>0$ and $u\in X$ are} arbitrary, we conclude that  $X \subset W^{1,p}_0(\Omega,a)$. This completes the proof of Lemma \ref{approxbyC1uptobdry}.
\end{proof}

We {complete} this subsection {with} some remarks on Sobolev embeddings for the space $W^{1,p}(\Omega,a)$. From \eqref{sobolevexponents} and H\"{o}lder's inequality, the embedding
$$
W^{1,p}(\Omega,a) \hookrightarrow W^{1,p_s}(\Omega)
$$
is continuous (see \cite{DRABEKKUFNERNICOLOSI}). Since $1<p_s<N$, from the standard Sobolev embeddings,
\begin{equation}\label{embeddingimportant}
W^{1,p}(\Omega,a)\hookrightarrow L^{m}(\Omega)
\end{equation}
with
$$
1\leq m \left\{
\begin{array}{ccc}
\leq p_s^*,&\hbox{continuous embedding}\\
\\
<p_s^*, & \hbox{compact embedding.}
\end{array}
\right.
$$

In particular, there exists $\mu_0>0$ such that for every $v\in W^{1,p}(\Omega,a)$,
\begin{equation}\label{WeightedSobolevEmbedding}
\mu_0 \|v\|_{L^{p_s^*}(\Omega)}\leq \|v\|_{W^{1,p}(\Omega,a)}.
\end{equation}

From \eqref{hyp:rma} $
p_s^*>p$, which together with \eqref{embeddingimportant}, yield
$$
W_0^{1,p}(\Omega,a)\hookrightarrow L^p(\Omega).
$$

\medskip
Friedrich's inequality implies that the norm
$$
\|u\|_{W_0^{1,p}(\Omega,a)}:=\left(\int_{\Omega}a(x)|\nabla u|^pdx\right)^{\frac{1}{p}}
$$
in $W^{1,p}_0(\Omega,a)$ is equivalent to the norm $\|\cdot\|_{W^{1,p}(\Omega,a)}$.

\subsection{A priori estimates and regularity} \label{2.2}

\medskip
Our next results concern with a priori estimates for solutions of \eqref{MODELDEGQUASILINEAR}.

\medskip
\begin{prop}\label{BrezisKato}
let (A1),(A2) and (B1) hold. Let $F(x, \zeta)$ be a Caratheodory function such that,  
\begin{equation}\label{growthF(x,u)}
|F(x,\zeta)|\leq b(x)\quad \quad \hbox{for a.e. }\,\, x\in \Omega, \quad  \zeta \in \R.
\end{equation}

There exists a constant $C>0$ such that for every solution $u \in W^{1,p}_0(\Omega,a)$ of \eqref{MODELDEGQUASILINEAR},
\begin{equation}\label{BrezisKatoLINFTY}
\|u\|_{L^{\infty}(\Omega)}\leq C\big({1+ \|b\|_{L^{1}(\Omega)}}+\|u\|_{L^p(\Omega)}\big).
\end{equation}
\end{prop}

The proof of Proposition \ref{BrezisKato} follows the lines of Theorem 3.2 in \cite{DRABEKKUFNERNICOLOSI} using Moser's iteration technique. We refer the reader also to Appendix B in \cite{STRUWE} and references therein for the case $p=2$.

\begin{proof} Consider first the case $p<q$. let $u\in W_0^{1,p}(\Omega,a)$ be a solution of \eqref{MODELDEGQUASILINEAR} and write $u= u^+ + u^-$.

\medskip
For $\ep>0$, denote 
$$
C(\ep):= \frac{(1 + \ep)^{p}}{1+\ep p}.
$$

Let $M, \ep>0$ be arbitrary and consider the function
$$
v_{M,\ep}(x):=\min(|u|^{1+\ep}, M), \quad x\in \Omega.
$$

From part {\it (iii)} in Lemma \ref{positivenegativeparts}, $v_{M,\ep}\in W^{1,p}(\Omega,a)$ and observe that
\begin{eqnarray*}
\int_{\Omega}a(x)|\nabla v_{M,\ep}|^pdx &=&
\int_{\{|u|^{1+\ep}\leq M\}}a(x)|\nabla \left(|u|^{1+\ep}\right)|^{p}dx\\
&=&(1+\ep)^p\int_{\{|u|^{1+\ep}\leq M\}}a(x)|u|^{p\ep}|\nabla u|^pdx\\
&=&\frac{(1+\ep)^p}{1+\ep p}\int_{\{|u|^{1+\ep}\leq M\}}a(x)|\nabla u|^{p-2}\nabla u \cdot \nabla(|u|^{p\ep}u)dx\\
&=&C(\ep)\int_{\Omega}a(x)|\nabla u|^{p-2}\nabla u \cdot \nabla\varphi_{M,\ep}dx
\end{eqnarray*}
where $\varphi_{M,\ep}=\min(|u|^{p\ep}u^+,M^{\frac{1 + p \ep}{1 + \ep}}) + \max(|u|^{p\ep}u^-,-M^{\frac{1 + p \ep}{1 + \ep}})$.

\medskip
Part {\it (ii)} in Lemma \ref{positivenegativeparts} yields that $u^+,u^-\in W^{1,p}_0(\Omega,a)$. Through approximation of $u^+,u^-$ in $W^{1,p}_0(\Omega,a)$ by $C^{\infty}_c(\Omega)$ functions, we verify that $\varphi_{M,\ep}\in W^{1,p}_0(\Omega,a)$. 

\medskip
Testing \eqref{weakaslnMODELDEGQUASILINEAR} against $\varphi_{M,\ep}$,
\begin{eqnarray*}
\int_{\Omega}a(x)|\nabla v_{M,\ep}|^p dx &=& C(\ep)\int_{\Omega} F(x,u)\varphi_{M,\ep}dx\\
&\leq& C(\ep)\int_{\Omega}|F(x,u)|\,|u|^{1+\ep p}dx.
\end{eqnarray*}

The last inequality and \eqref{growthF(x,u)} yield
\begin{eqnarray}
\int_{\Omega}a(x)|\nabla v_{M,\ep}|^p dx &\leq &C(\ep)\int_{\Omega}b(x)|u|^{1+\ep p}dx.\label{nablavMep}
\end{eqnarray}

\medskip
To estimate the right hand side in \eqref{nablavMep} we proceed as follows: 
\begin{eqnarray}
\int_{\Omega}b(x)|u|^{1 + \ep p}dx &=& \int_{\{|u|\leq 1\}}b(x)|u|^{1 + \ep p}dx + \int_{\{|u|>1\}}b(x)|u|^{1 + \ep p}dx   
\nonumber\\
&\leq&\int_{\Omega}b(x)dx + \int_{\{|u|>1\}}b(x)|u|^{1 + \ep p} \,|u|^{p-1}dx
\nonumber\\
&\leq&\|b\|_{L^1(\Omega)} +\underbrace{\int_{\Omega}b(x)|u|^{p(1 +\ep)}dx}_{I}
\label{estimateI}.
\end{eqnarray}

\medskip
Next we estimate $I$. Let $L>0$ be arbitrary. From {\it (B1)} and H\"{o}lder's inequa\-lity
\begin{eqnarray}
I &=& \int_{\{b\leq L\}}b(x)|u|^{p(1+\ep)}dx +\int_{\{b>L\}}b(x)|u|^{p(1+\ep)}dx \nonumber\\
&\leq&
L\int_{\Omega} |u|^{p(1+\ep)}dx +\left(\int_{\Omega}b(x)^{\frac{q}{q-p}}dx\right)^{\frac{q-p}{q}}\left(\int_{\{b>L\}}|u|^{q(1+\ep)}dx\right)^{\frac{p}{q}}\nonumber\\
&\leq& L\int_{\Omega}|u|^{p(1+\ep)}dx + \|b\|_{L^{\frac{q}{q-p}}(\Omega)}
{\rm meas}\{b>L\}^{{\frac{p(p_s^* - q)}{qp_s^*}}}
\left(\int_{\Omega}|u|^{p_s^*(1+\ep)}dx
\right)^{\frac{p}{p_s^*}} \label{estimateII}.
\end{eqnarray}

\medskip
Setting
$$
\sigma(L):=
\|b\|_{L^{\frac{q}{q-p}}(\Omega)}{\rm meas}\{b>L\}^{{\frac{p(p_s^* - q)}{qp_s^*}}}, \qquad c_0:=1 + \|b\|_{L^{1}(\Omega)}
$$ 
and putting together \eqref{nablavMep}, \eqref{estimateI} and  \eqref{estimateII}, we obtain
\begin{equation}\label{estimateaMepsilon}
\int_{\Omega}a(x)|\nabla v_{M,\ep}|^pdx\leq C(\ep)\left(c_0 + L\int_{\Omega}|u|^{p(1+\ep)}dx +\ \sigma(L)\left(
\int_{\Omega}|u|^{p_s^*(1+\ep)}dx\right)^{\frac{p}{p_s^*}}
\right).
\end{equation}

Notice that the right hand side in \eqref{estimateaMepsilon} does not depend on $M>0$. Using the {\it Monotone Convergence Theorem} we pass to the limit as $M\to \infty$ to find that
\begin{equation}\label{estimateaMepsilon2}
\int_{\Omega}a(x)|\nabla (|u|^{1 +\ep})|^pdx\leq C(\ep)\left(c_0 + L\int_{\Omega}|u|^{p(1+\ep)}dx +\ \sigma(L)\left(
\int_{\Omega}|u|^{p_s^*(1+\ep)}dx\right)^{\frac{p}{p_s^*}}
\right)
\end{equation}
which implies that $|u|^{1+\ep}\in W^{1,p}(\Omega,a)$, provided that $|u|^{1+\ep}\in L^{p_s^*}(\Omega)$.

\medskip

Using that $p<q<p_s^*$, set 
$$
{\kappa_0:=  \frac{p(p_s^*-q)}{p_s^*(q-p)}}\quad \hbox{and}\quad c_1:=\|b\|^{1 + \kappa_0}_{L^{\frac{q}{q-p}}(\Omega)}>0.
$$

We estimate
\begin{eqnarray}
\sigma(L)&\leq&
\|b\|_{L^{\frac{q}{q-p}}(\Omega)}\left(\int_{\{b>L\}}L^{\frac{q}{q-p}}dx\right)^{{\frac{p(p_s^* - q)}{qp_s^*}}}L^{-\frac{p(p_s^*-q)}{p_s^*(q-p)}}\nonumber\\
&\leq& \|b\|^{1 + \frac{p(p_s^*-q)}{p_s^*(q-p)}}_{L^{\frac{q}{q-p}}(\Omega)}
L^{-\frac{p(p_s^*-q)}{p_s^*(q-p)}}
\nonumber\\
&=&c_1 \,L^{-\kappa_0}.\label{doublenumeral}
\end{eqnarray}

\medskip

Next, fix $\D>0$ such that $0<\mu_0^p - \D^p<1$, where $\mu_0>0$ is the Sobolev constant in \eqref{WeightedSobolevEmbedding}. Choose $L=L_{\ep}>0$ such that
$$
c_{1}C(\ep)\,L^{-\kappa_0}= \D^p
$$
so that from \eqref{WeightedSobolevEmbedding}, \eqref{estimateaMepsilon2} and \eqref{doublenumeral}, we get
\begin{equation}\label{EstimateIII}
\begin{aligned}
\mu_0^p\left(
\int_{\Omega}|u|^{p_s^*(1+\ep)}dx\right)^{\frac{p}{p_s^*}}&\leq \int_{\Omega}a(x)|\nabla\left(|u|^{1+\ep}\right)|^p dx\\
&\leq C(\ep)\left(c_0 + L\int_{\Omega}|u|^{p(1+\ep)}dx
\right) + \ \D^p\left(
\int_{\Omega}|u|^{p_s^*(1+\ep)}dx\right)^{\frac{p}{p_s^*}}.
\end{aligned}
\end{equation}

\medskip
Since $L=\frac{c_1^{\frac{1}{\kappa_0}}C^{\frac{1}{\kappa_0}}(\ep)}{\D^{\frac{1}{\kappa_0}}}$, \eqref{EstimateIII} yields
$$
(\mu_0^p-\D^p)\|u\|^{p(1+\ep)}_{L^{p_s^*(1+\ep)}(\Omega)} \leq \frac{c_1^{\frac{1}{\kappa_0}}}{\D^{\frac{p}{\kappa_0}}}C(\ep)^{1+\frac{1}{\kappa_0}}\left(\frac{c_0}{L}+\|u\|^{p(1+\ep)}_{L^{p(1+\ep)}(\Omega)}\right).
$$

\medskip
Using our choice of $L=L_{\ep}$ and setting $\kappa_1:=1+\frac{1}{\kappa_0}$, there exists $c_2> 0$ such that
\begin{equation*}
\|u\|_{L^{p_s^*(1+\ep)}(\Omega)}\leq c_2^{\frac{1}{p(1+\ep)}}C(\ep)^{\frac{\kappa_1}{p(1+\ep)}}\left[\left(\frac{c_0}{C(\ep)^{\frac{1}{\kappa_0}}}\right)^{\frac{1}{p(1+\ep)}}+\|u\|_{L^{p(1+\ep)}(\Omega)}\right].
\end{equation*}

\medskip
Notice that $c_0>1$ and $C(\ep)\geq 1$. Therefore,\begin{equation}\label{iterativeetimate}
\|u\|_{L^{p_s^*(1+\ep)}(\Omega)} \leq c_2^{\frac{1}{p(1+\ep)}}C(\ep)^{\frac{k_1}{p(1+\ep)}}\left(c_0+\|u\|_{L^{p(1+\ep)}(\Omega)}\right).
\end{equation}
\medskip
To finish the proof,{ we use \eqref{iterativeetimate} and proceed as in the proof of Lemma 3.2 in \cite{DRABEKKUFNERNICOLOSI}. For the sake of completeness we include the details.}

\medskip
We run an iterative scheme as follows. Set $\ep_0=0$ to find from \eqref{iterativeetimate} that 
\begin{equation}\label{iterativeetimate1}
\|u\|_{L^{p^*_s}(\Omega)}\leq c_2^{\frac{1}{p}}(c_0
+\|u\|_{L^{p}(\Omega)}).
\end{equation}

\medskip
Next, setting $\ep_1>0$ such that $p(1+\ep_1)=p_s^*$ we get from \eqref{iterativeetimate} and \eqref{iterativeetimate1}
\begin{eqnarray*}
\|u\|_{L^{p_s^*(1+\ep_1)}(\Omega)} &\leq& c_2^{\frac{1}{p(1+\ep_1)}}C(\ep_1)^{\frac{k_1}{p(1+\ep_1)}}\left(c_0+\|u\|_{L^{p(1+\ep_1)}(\Omega)}\right)\\
&\leq& c_2^{\frac{1}{p(1+\ep_1 )}}C(\ep_1)^{\frac{\kappa_1}{p(1+\ep_1)}}(c_0  + c_2^{\frac{1}{p}}(c_0
+\|u\|_{L^{p}(\Omega)}))
\end{eqnarray*}
and since $C(\ep_0)=1$, 
\begin{equation}\label{iterativeetimate2}
\|u\|_{L^{p_s^*(1+\ep_1)}(\Omega)}\leq c_2^{\frac{1}{p(1+\ep_0)}+\frac{1}{p(1+\ep_1)}}C(\ep_0)^{\frac{\kappa_1}{p(1+\ep_0)}}C(\ep_1)^{ \frac{\kappa_1}{p(1+\ep_1)}}\left(c_0(1+c_2^{-\frac{1}{p}}) +\|u\|_{L^p(\Omega)}\right).
\end{equation}

\medskip
For the next iteration we choose $\ep_2$ given by $p(1+\ep_2)=p_s^*(1+\ep_1)$ and from \eqref{iterativeetimate} and \eqref{iterativeetimate2},
$$
\|u\|_{L^{p_s^*(1+\ep_2)}(\Omega)} \leq c_2^{\frac{1}{p(1+\ep_2)}}C(\ep_2)^{\frac{k_1}{p(1+\ep_2)}}\left(c_0+\|u\|_{L^{p(1+\ep_2)}(\Omega)}\right)
$$
$$
\begin{aligned}
\leq & c_2^{\frac{1}{p(1+\ep_2)}}C(\ep_2)^{\frac{\kappa_1}{p(1+\ep_2)}}\left[c_0+  c_2^{\frac{1}{p(1+\ep_0)}+\frac{1}{p(1+\ep_1 )}}C(\ep_0)^{\frac{\kappa_1}{p(1+\ep_0)}}C(\ep_1)^{ \frac{\kappa_1}{p(1+\ep_1)}}\left(c_0(1+c_2^{-\frac{1}{p}}) +\|u\|_{L^p}\right)\right]
\\
\leq &c_2^{\sum_{j=0}^{2}\frac{1}{p(1+\ep_j)}}\Pi_{j=0}^{2}C(\ep_j)^{\frac{\kappa_1}{p(1+\ep_j)}}\left(c_0(1+c_2^{-\frac{1}{p}}+ c_2^{\sum_{j=0}^{1}\frac{-1}{p(1+\ep_j)}}\Pi_{j=0}^{1}C(\ep_j)^{\frac{-\kappa_1}{p(1+\ep_j)}}) +\|u\|_{L^p}\right)
\end{aligned}
$$
so that 
$$
\|u\|_{L^{p_s^*(1+\ep_2)}}\leq c_2^{\sum_{j=0}^{2}\frac{1}{p(1+\ep_j)}}\Pi_{j=0}^{2}C(\ep_j)^{\frac{\kappa_1}{p(1+\ep_j)}}\left(c_0(1+c_2^{-\frac{1}{p}}+ c_2^{\frac{-2}{p}}) +\|u\|_{L^p}\right).
$$

\medskip
We proceed inductively by choosing $\ep_n>0$ such that
$$
p(1+\ep_n) := p_s^*(1+\ep_{n-1}), \quad \hbox{i.e.} \quad \frac{1}{1+\ep_n}= \left(\frac{p}{p_s^*}\right)^{n}, \quad n\geq 0
$$
to obtain that
\begin{equation}\label{bootstrapinequality}
\|u\|_{L^{p_s^*(1+\ep_n)}} \leq c_2^{\frac{1}{p}\sum_{j=0}^{n}\frac{1}{1+\ep_j}}\Pi_{j=0}^{n}C(\ep_j)^{\frac{\kappa_1}{p(1+\ep_j)}}\left(c_0\sum_{m=0}^n c_2^{\frac{-m}{p}}+\|u\|_{L^p}\right).
\end{equation}

\medskip
Since $p<p_s^*$,
$$
\sum_{j=0}^{\infty}\frac{1}{1+\ep_j}=\sum_{j=0}^{\infty}\left(\frac{p}{p_s^*}\right)^{j}=\frac{p_s^*}{p_s^*-p}<\infty
$$
and 
\begin{eqnarray*}
0<\log\left(\Pi_{j=0}^{n}C(\ep_j)^{\frac{\kappa_1}{p(\ep_j+1)}}\right)&=&\sum_{j=0}^n\frac{\kappa_1}{\ep_j +1} \log\left(\frac{1+\ep_j}{(1+p\,\ep_j)^{\frac{1}{p}}}\right)
\\
&\leq& \sum_{j=0}^n\frac{\kappa_1}{\ep_j +1} \log\left((1+\ep_j)^{1-\frac{1}{p}}\right)
\\
&=& \kappa_1\frac{p-1}{p}\log\left(\frac{p_s^*}{p}\right)\sum_{j=0}^nj\left(\frac{p}{p_s^*}\right)^{j}.
\end{eqnarray*}

Testing the root criterion for convergence of series,
$$
\liminf_{j\to\infty}\sqrt[j]{j\left(\frac{p}{p_s^*}\right)^{j}}\leq \frac{p}{p_s^*}<1.
$$ 

\medskip

Consequently, in the limit of the iterative argument the sequences
$$
c_2^{\frac{1}{p}\sum_{j=1}^{n}\frac{1}{\ep_j +1}}, \quad \Pi_{j=0}^{n}C(p,\ep_j)^{\frac{\kappa_1}{p(\ep_j+1)}},\quad c_0\sum_{m=0}^n c_2^{\frac{-m}{p}} \quad \big(c_2>2, \quad p>1\big)
$$ 
are uniformly bounded. Passing to the limit as $n\to \infty$ in \eqref{bootstrapinequality} and recalling $c_0=1+\|b\|_{L^1(\Omega)}$, we
obtain
$$
\|u\|_{L^{\infty}(\Omega)}\leq C\big( 1+\|b\|_{L^1(\Omega)} + \|u\|_{L^p(\Omega)}\big).
$$

\medskip
The case $p=q$, i.e.,$\frac{q}{q-p}=\infty$, follows by choosing  any $\tilde{q}$ with $p<\tilde{q}<p_s^*$ and performing the above estimates, changing $q$ by $\tilde{q}$ where necessary, and using that 
$$
\|b\|_{L^{\frac{\tilde{q}}{\tilde{q}-p}}(\Omega)}\leq {\rm meas}(\Omega)^{\frac{\tilde{q}}{p}}\|b\|_{L^{\infty}(\Omega)}
$$
and this completes the proof of Proposition 
\ref{BrezisKato}.
\end{proof}

%\medskip
%As a consequence of the proof of Lemma \ref{BrezisKato}, if condition \eqref{growthF(x,u)} is replaced by
%$$
%|F(x,\zeta)| \leq b(x)|\zeta|^{p-1} \quad \hbox{for a.e.} \quad x\in \Omega, \,\,\zeta \in \R, 
%$$
%then the estimate \eqref{BrezisKatoLINFTY} reads as
%$$
%\|u\|_{L^{\infty}(\Omega)} \leq C\|u\|_{L^p(\Omega)}.
%$$

Solutions to \eqref{MODELDEGQUASILINEAR} are locally H\"{o}lder continuous. This is the content of the next proposition.

\begin{prop}\label{SerrinRegul}
Let (A1)-(A2) and (B1) hold. Assume that $F(x, \zeta)$ is a Caratheodory function satisfying 
\eqref{growthF(x,u)0}.

\medskip
Let $D \subset \R^N$ be an open set such that $\overline{D} \subset \Omega$. Then, there exist $\alpha \in (0,1)$ and $\kappa>0$, both depending only on $D,a,b,p,q,N$, and there exists $C=C(\Omega, D)>0$ such that for any weak solution $u\in W^{1,p}_{0}(\Omega,a)\cap L^{\infty}(\Omega)$ of \eqref{MODELDEGQUASILINEAR}, 
\begin{itemize}
\item[(i)] $u\in C^{0,\alpha}(\overline{D})$ and 

\item[(ii)] for any $x,y \in D$,
$$
|u(x)-u(y)| \leq C\big[\kappa + \|u\|_{L^{\infty}(D)}\big]|x-y|^{\alpha}.
$$
\end{itemize}
\end{prop}

\begin{proof}
This result is a direct consequence of the inequalities in \eqref{exponentsinequalities} and Theorems 1 and 8 in \cite{SERRIN}.
\end{proof}

\medskip

\subsection{Boundary behavior of solutions to \eqref{MODELDEGQUASILINEAR}}\label{2.3}

Next, we review some well-known facts about Fermi coordinates. We refer the reader to \cite{DOCARMOMANFREDO,MANASSEMISNER} for more details on the developments hereby presented.

\medskip
For any $\rho>0$, set
\begin{equation}\label{def:Omegarho}
\Omega_{\rho}:=\{x\in \Omega\,:\,0<\rdist(x,\pp \Omega)<\rho\}
\end{equation}
and let ${\rm n} :\pp \Omega \to S^{N-1}$ be the {\it inner unit normal vector} to $\pp \Omega$.

\medskip

In what follows, we assume with no loss of generality that the mapping
\begin{equation*}\label{diffeom}
{\rm x}(\y,y_N):= \y + y_N{\rm n}(\y) \quad \hbox{for}\quad (\y,y_N)\in \pp \Omega\times (0,\rho_0)
\end{equation*}
is a smooth diffeomorphism onto $\Omega_{\rho_0}$ (see \eqref{def:distfunc}).

\medskip

Let $\Phi: \U \to \Phi(\U)$ be a local parametrisation of $\pp \Omega$ with $\U\subset \R^{N-1}$ being an open connected set. 

\medskip

For $(y,y_N)\in \mathcal{U}\times (0,\rho_0)$ define 
$$
\begin{aligned}
X(y,y_N)\,:=& \,{\rm x}(\Phi(y),y_N)
\,= \, \Phi(y) + y_N{\rm n}(y),
\end{aligned}
$$
where, abusing the notation, we have written ${\rm n}(y):={\rm n}(\Phi(y))$.

\medskip
Set $g_{ij}:=\pp_{y_i}X\cdot \pp_{y_j}X$ for $i,j=1,\ldots,N$. Observe that
$$
\begin{aligned}
\forall i,j=1,\ldots,N-1&:  
&g_{ij} = &\pp_{y_i} \Phi \cdot \pp_{y_j} \Phi + 2y_N \pp_{y_i} \Phi \cdot \pp_{y_j} {\rm n} + y_N^2 \pp_{y_i} {\rm n} \cdot \pp_{y_j} {\rm n},
\\
\forall i=1,\ldots,N-1&: & g_{Ni}  = & g_{iN}=0  
\end{aligned}
$$
and $g_{NN}=1$.

\medskip
Therefore, setting
$$
g:=\left(
\begin{array}{cc}
(g_{ij})_{N-1 \times N-1}&0\\
0&1
\end{array}
\right) \quad \hbox{we have} \quad g^{-1}:=\left(
\begin{array}{cc}
(g^{ij})_{N-1 \times N-1}&0\\
0&1
\end{array}
\right).
$$

Here $g= (g_{ij})_{N\times N}$ is the induced Riemann metric on $X(\U\times (0,\rho_0))$ with inverse $g^{-1}=(g^{ij})_{N\times N}$. Since $\Omega$ is a smooth domain the matrices $g$ and $g^{-1}$ are smooth in $\partial \Omega \times (-\rho_0,\rho_0)$.

\medskip
Next, we observe that from {\it (A1)}, {\it (A2)} and from the change of variables it follows that
$$
\int_{\Omega_{\rho_0}} a(x)dx =\int_{\partial\Omega}\int_0^{\rho_0}  {\sf a}(y_N) \, \sqrt{\det g} dy_N  \, dy \leq C_{\Omega}\int_0^{\rho_0}  {\sf a}(y_N) \, dy_N 
$$ 
for some constant $C_{\Omega}>0$. We thus conclude that $a\in L^1(\Omega)$. Similarly, we can show that $a^{-1}\in L^s(\Omega)$. This proves the claim made at the beginning of subsection \ref{2.1}.

\medskip
In the next lemma we compute $\rdiv\left(a(x)|\nabla w|^{p-2}\nabla w\right)$ in the coordinates $x=X(y,y_N)$  for certain class of functions $w$.

\medskip
First, for any $\varphi \in W^{1,1}_{loc}(\Omega_{\rho_0})$, we write 
$$
x={X}(y,y_N) \quad \hbox{and} \quad \varphi(x):=\varphi(y,y_N) \quad \hbox{for} \quad (y,y_N)\in \mathcal{U} \times (0,\rho_0).
$$

Notice that in the coordinates $x=X(y,y_N)$, 
\begin{equation*}
\nabla_{x} \varphi(x) =  \sum_{i,k=1}^{N-1}g^{ik}\,\pp_{y_k}\varphi(y,y_N)\pp_{y_i}X + \pp_{y_N}\varphi(y,y_N){\rm n}(y).
\end{equation*}

\begin{lemma}
\label{pLapincoordinates}
Let (A1) and (A2) hold and let {${w}\in W^{1,p}(\Omega_{\rho_0},a)$} be such that $w(x):= {w}({\rdist(x,\pp \Omega)})$ in ${\Omega}_{\rho_0}$. Then, 
\begin{itemize}
\item[(i)] for any $\varphi \in C^{\infty}_c(\Omega_{\rho_0})$,
\begin{equation}\label{pLapincoordinatesintgform}
\begin{aligned}
\int_{\Omega_{\rho_0}} a(x)|\nabla w|^{p-2}\nabla w \cdot \nabla \varphi dx &\\=&
&\int_{\pp \Omega}\int_0^{\rho_0}  {\sf a}(y_N)|\pp_{y_N} w|^{p-2}\pp_{y_N} w \, \pp_{y_N} \varphi(y,y_N) \, \sqrt{\det g} dy_N \, dy.
\end{aligned}
\end{equation}

\medskip

\item[(ii)] If in addition ${\sf a}(\y_N)|\pp_{\y_N}{w}|^{p-2}\pp_{\y_N}{w}$ is absolutely continuous in $(0,\rho_0)$,
\begin{equation}
\label{pLapincoordinatesform}
\begin{aligned}
\rdiv \left(a(x) |\nabla { w}|^{p-2}\nabla { w}\right) & =\pp_{\y_N}\left({\sf a}(\y_N)|\pp_{\y_N}{w}|^{p-2}\pp_{\y_N}{w}\right)\\
& \quad + \pp_{\y_N}\log\left(\sqrt{\det g}\right){\sf a}(\y_N)|\pp_{\y_N}{w}|^{p-2}\pp_{\y_N}{w},
\end{aligned}
\end{equation}
for a.e. $x=X(y,y_N)$ with $(y,y_N)\in \U \times (0,\rho_0)$.
\end{itemize}
\end{lemma}

\begin{proof} Identity \eqref{pLapincoordinatesintgform} is a direct consequence of {\it (A2)} and the {change of variables}. 

\medskip
Assume now that ${\sf a}(y_N)|\pp_{y_N}{w}|^{p-2}\pp_{y_N}{w}$ is absolutely continuous in $(0,\rho_0)$. Integrating by parts in \eqref{pLapincoordinatesintgform},   
\begin{equation}\label{p-lapdistrib}
\rdiv \left(a(x) |\nabla w|^{p-2}\nabla w\right)=\frac{1}{\sqrt{\det g}}\pp_{y_N} \left(\sqrt{\det g}\,{\sf a}(y_N)\,|\pp_{y_N}w|^{p-2}\pp_{y_N}w\right).
\end{equation}

\medskip
Performing the product differentiation in the right-hand side of \eqref{p-lapdistrib}, equality \eqref{pLapincoordinatesform} follows. This completes the proof of the lemma.
\end{proof}

Denote
\begin{equation}\label{distweight}
{\rm d}(x):= \int_{0}^{\rdist(x,\partial \Omega)} {\sf a}^{-\frac{1}{p-1}}(\zeta)d \zeta \quad \hbox{for} \quad x\in \Omega_{\rho_0}.
\end{equation}

\begin{lemma}\label{comparison solution}
Assume (A1) and (A2). Then there exists a function $\psi \in W^{1,p}(\Omega,a)\cap C(\overline{\Omega})$ with $\psi|_{{\pp \Omega}}=0$ and such that
\begin{itemize}
\item[(i)] for every $\varphi \in W^{1,p}_0(\Omega_{\rho_0},a)$ with $\varphi \geq 0$ a.e. in $\Omega_{\rho_0}$,
\begin{equation}\label{integralIdentcomparison}
\int_{\Omega_{\rho_0}}a(x)|\nabla \psi|^{p-2}\nabla \psi \cdot \nabla \varphi dx \geq \int_{\Omega_{\rho_0}} a^{ -\frac{1}{p-1}}(x)\varphi dx
\end{equation}
and 

\medskip
\item[(ii)] there exists a constant ${C}>0$ such that for every $x\in \Omega_{\rho_0}$
\begin{equation}\label{bdryestimate}
\frac{1}{{C}}{\rm d}(x)\leq \psi(x)\leq {C}{\rm d}(x).
\end{equation}
\end{itemize}
\end{lemma}

\begin{proof} 
Consider local coordinates $\Phi_i: \mathcal{U}_i \to \Phi_i(\mathcal{U}) \subset \partial \Omega$ with $\{\Phi_i(\U_i)\}_{i=1}^m$ being a finite covering for $\pp \Omega$. Using the geometric notations from Lemma \ref{pLapincoordinates}, set 
\begin{equation}\label{def:LambdaCapital}
\Lambda:=\max \limits_{1\leq i\leq m} \max \limits_{\substack{y\in \overline{\mathcal{U}_i},\\
0\leq y_N\leq \rho_0}}\left|\pp_{y_N}\log\left(\sqrt{\det g}\right)\right|<\infty.
\end{equation}

Next, let $A>0$ be a fixed constant such that
$$
A\,>\, \int_{0}^{\rho_0}e^{\Lambda \tau}{\sf a}^{-\frac{1}{p-1}}(\tau)d\tau
$$
and consider the function ${\uppsi}:[0, \infty)\to [0,\infty)$ defined by
\begin{equation}\label{integralsupersolution}
{\uppsi}(y_N)= \int_{0}^{y_N}\frac{1}{e^{\frac{\Lambda \zeta}{p-1}}\,{\sf a}^{\frac{1}{p-1}}(\zeta)}\left[A - \int_{0}^{\zeta}e^{\Lambda \tau}{\sf a}^{-\frac{1}{p-1}}(\tau)d\tau\right]^{\frac{1}{p-1}}d\zeta \quad \hbox{for} \quad y_N\in [0,\rho_0].
\end{equation}

Next, we list some properties of the function ${\uppsi}$ which are direct consequences of \eqref{hyp:rma} and \eqref{integralsupersolution}: 
\begin{itemize}
\item[(I)] $\uppsi\in C[0,\rho_0]$ and $\uppsi(0)=0$;

\medskip
\item[(II)] $\uppsi$ is differentiable in $(0,\rho_0]$ with$$
\partial_{y_N}\uppsi(y_N)={e^{-\frac{\Lambda y_N}{p-1}}\,{\sf a}^{-\frac{1}{p-1}}(y_N)}\left[A - \int_{0}^{y_N}e^{\Lambda \tau}{\sf a}^{-\frac{1}{p-1}}(\tau)d\tau\right]^{\frac{1}{p-1}}\quad \hbox{for} \quad y_N\in (0,\rho_0);
$$

\item[(III)] ${\sf a}(y_N)|\partial_{y_N} \uppsi|^{p-1}$ is absolutely continuous in $(0,\rho_0)$ and
\begin{equation}\label{ddd}
-\pp_{y_N}\left({\sf a}(y_N)|\pp_{y_N} \uppsi|^{p-1}\right) - \Lambda {\sf a}(y_N)|\pp_{y_N} \uppsi|^{p-1}\,=\,{\sf a}^{\frac{1}{p-1}}(y_N) \quad \hbox{in} \quad (0,\rho_0).
\end{equation}
\end{itemize}

Next, observe that a direct calculation yields the estimate\begin{equation}
\label{est:Eqn_psi1}
\begin{aligned}
-\frac{1}{\sqrt{\det g}}\pp_{y_N}\left(\sqrt{\det g}\,{\sf a}(y_N) |\pp_{y_N} \uppsi|^{p-1}\right)&\geq   -\pp_{y_N}\left(\,{\sf a}(y_N) |\pp_{y_N} \uppsi|^{p-1}\right)\\
& \hspace{50pt}
 - \Bigl|\partial_{y_N} \log
\Bigl(
\sqrt{{\rm det}g}
\Bigr)\Bigr|{\sf a}(y_N) |\pp_{y_N} \uppsi|^{p-1}\\
& \geq -\pp_{y_N}\left({\sf a}(y_N)|\pp_{y_N} \uppsi|^{p-1}\right) - \Lambda {\sf a}(y_N)|\pp_{y_N} \uppsi|^{p-1}
\end{aligned}
\end{equation}
in $\U_i\times (0,\rho_0)$  for every $i=1,\ldots,m$, where the latter inequality in \eqref{est:Eqn_psi1} folllows from the definition of $\Lambda$ in \eqref{def:LambdaCapital}.

\medskip
The estimate in \eqref{est:Eqn_psi1} and the equality in \eqref{ddd} imply that \begin{equation}
\label{est:Eqn_psi}
\begin{aligned}
-\frac{1}{\sqrt{\det g}}\pp_{y_N}\left(\sqrt{\det g}\,{\sf a}(y_N) |\pp_{y_N} \uppsi|^{p-1}\right)& \geq {\sf a}^{-\frac{1}{p-1}}(y_N) \quad \hbox{in}\quad \U_i\times (0,\rho_0)
\end{aligned}
\end{equation}
for every $i=1,\ldots,m$.

\medskip

Observe also that (II) yields that
$$
\frac{1}{{C}} {\sf a}^{-\frac{1}{p-1}}(y_N)\leq \pp_{y_N}\uppsi(y_N)\leq  {C} {\sf a}^{-\frac{1}{p-1}}(y_N) \quad \hbox{for} \quad y_N\in (0,\rho_0),
$$
where
$$
C:=\max \Bigl\{A, e^{\frac{\Lambda \rho_0}{p-1}}\left(A - \int_{0}^{\rho_0}e^{\Lambda \tau}{\sf a}^{-\frac{1}{p-1}}(\tau)d\tau\right)^{-1}\Bigr\}>0.
$$

Consequently,
\begin{equation}\label{decaysupersol}
\frac{1}{{C}}\,\int_{0}^{y_N}{\sf a}^{-\frac{1}{p-1}}(\zeta)d\zeta  \leq  \uppsi(y_N)\leq  {C} \,\int_{0}^{y_N}{\sf a}^{-\frac{1}{p-1}}(\zeta)d\zeta
\end{equation}
for every $y_N \in (0,\rho_0)$.

\medskip
{Set $\psi(x):= \uppsi({\rm dist}(x, \pp \Omega))$ for $x\in \overline{\Omega}_{\rho_0}$ (see  \eqref{def:distfunc}) and extend it continuously as $\psi(x):=\uppsi(\rho_0)$ in $\Omega-\Omega_{\rho_0}$. From {(I)} and {(II)}, it is direct to verify that $\psi \in  W^{1,p}(\Omega,a)\cap C(\overline{\Omega})$, $\psi>0$ in $\Omega$ and $\psi|_{\pp \Omega}=0$.}

\medskip

\medskip
Estimate \eqref{est:Eqn_psi} and part {(ii)} from Lemma \ref{pLapincoordinates}, imply that
\begin{equation}
\label{inequalitysupersolutiomn1}
-{\rm div}\Bigl(a(x)|\nabla \psi|^{p-2}\nabla \psi\Bigl) \geq a^{-\frac{1}{p-1}}(x) \quad \hbox{in} \quad \Omega_{\rho_0}.
\end{equation}
 
To prove {(i)} we use \eqref{inequalitysupersolutiomn1} and Lemma \ref{pLapincoordinates} to find, after integrating by parts, that for every $\varphi \in C^{\infty}_c(\Omega_{\rho_0})$ with $\varphi \geq 0$ in $\Omega_{\rho_0}$,
\begin{equation}\label{integralIdentcomparison1}
\int_{\Omega_{\rho_0}}a(x)|\nabla \psi|^{p-2}\nabla \psi \cdot \nabla \varphi dx \geq \int_{\Omega_{\rho_0}} a^{-\frac{1}{p-1}}(x)\varphi dx.
\end{equation}

The {\it Dominated convergence Theorem} implies that \eqref{integralIdentcomparison1} holds true for every $\varphi \in W^{1,p}_0(\Omega_{\rho_0},a)$.  
 
\medskip
Finally, the proof of the inequalities \eqref{bdryestimate} in {(ii)} readily follows from the definition of $\psi(x)$, the definition of ${\rm d}(x)$ in \eqref{distweight} and the inequalities in \eqref{decaysupersol}. This completes the proof of the lemma. 
\end{proof}

\medskip
Our final result in this section deals with boundary behavior of solutions of \eqref{MODELDEGQUASILINEAR} under an assumption related to \eqref{growthF(x,u)0}.

\medskip

\begin{prop}\label{BDCONDITION}
Assume (A1) and (A2) and let $\rho\in (0,\rho_0]$ be fixed. Let ${F}:\Omega \times \R \to \R$ be a Caratheodory function such that 
\begin{equation}\label{growthF(x,u)0stronger}
\sup \limits_{{x\in \Omega_{\rho},\,\,\zeta \in \R}} a^{\frac{1}{p-1}}(x) |{F}(x,\zeta)| < + \infty.
\end{equation}

There exists a constant $\tilde{C}>0$ such that for every solution $u\in W_0^{1,p}(\Omega,a) \cap L^{\infty}(\Omega)$ of \eqref{MODELDEGQUASILINEAR},
\begin{equation}\label{behavioru(x)}
\|{\rm d}^{-1}u\|_{L^{\infty}(\Omega_{\rho})}\leq \tilde{C}\,\big(1 + \|u\|_{L^{\infty}(\Omega)}\big).
\end{equation}
\end{prop}

\begin{proof}
Let $u \in W_0^{1,p}(\Omega,a)\cap L^{\infty}({\Omega})$ be a solution of \eqref{MODELDEGQUASILINEAR}. From \eqref{growthF(x,u)0stronger} there exists $\mathcal{K}>0$ such that 
\begin{equation}\label{reductionH4.1}
|{F}(x,u(x))|\,\leq \mathcal{K}\,a^{-\frac{1}{p-1}}(x)\,\quad \hbox{a.e. } x\in \Omega_{\rho}. 
\end{equation}

Let $\psi\in W^{1,p}_0(\Omega,a)\cap C(\overline{\Omega})$ be the function predicted in Lemma \ref{comparison solution}. We compare the solution $u$ of \eqref{MODELDEGQUASILINEAR}
with an appropriate multiple of $\psi$ in $\Omega_{\rho}$.

\medskip 
Fix $M>0$ such that
\begin{equation}\label{choiceM}
M^{p-1} > \mathcal{K}  \quad \hbox{and} \quad M \int_{0}^{\rho}{\sf a}^{-\frac{1}{p-1}}(\zeta)d\zeta \geq {C},
\end{equation}
where ${C}$ is the constant in \eqref{bdryestimate}. 
 
\medskip

Set $\overline{u}(x) = M\big(1+\|u\|_{L^{\infty}(\Omega)}\big)\,\psi(x)$ for $x\in \Omega$. We claim that $(u -\overline{u})^+ \in W^{1,p}_0(\Omega_{\rho},a)$. Let us assume for the moment that this claim has been proven. We would then proceed as follows. 

\medskip
Extend the functions in $W^{1,p}_0(\Omega_{\rho},a)$ to $\Omega_{\rho_0}$ and the functions in $W^{1,p}_0(\Omega_{\rho_0},a)$ to $\Omega$ by defining them as zero in $\Omega_{\rho_0}-\Omega_{\rho}$ and $\Omega-\Omega_{\rho_0}$, respectively. We obtain the embeddings:
$$
W^{1,p}_0(\Omega_{\rho},a)\hookrightarrow W^{1,p}_0(\Omega_{\rho_0},a)  \hookrightarrow W^{1,p}_0(\Omega,a).
$$

Next, choosing $\varphi=(u-\overline{u})^+$ in  \eqref{weakaslnMODELDEGQUASILINEAR} and \eqref{integralIdentcomparison}  and using the monotonicity of the vector function 
$$
\R^N \ni z\mapsto |z|^{p-2}z
$$
together with \eqref{reductionH4.1}, we get
$$
\begin{aligned}
0 \geq & \int_{\Omega_{\rho}\cap \{u \geq \overline{u}\}}a(x)\left(|\nabla \overline{u}|^{p-2}\nabla \overline{u} - |\nabla u|^{p-2}\nabla u\right)\cdot \nabla (u-\overline{u}) dx\\
= & \int_{\Omega_{\rho}}a(x)\left(|\nabla \overline{u}|^{p-2}\nabla \overline{u} - |\nabla u|^{p-2}\nabla u\right)\cdot \nabla (u-\overline{u})^+ dx\\
\geq & \int_{\Omega_{\rho}}\left(M^{p-1}a^{-\frac{1}{p-1}}(x)-F(x,u)\right)(u-\overline{u})^+dx 
\\
\geq & \int_{\Omega_{\rho}}  \left(M^{p-1}- \mathcal{K} \right)a^{-\frac{1}{p-1}}(x)(u-\overline{u})^+ dx.
\end{aligned}
$$

\medskip
From \eqref{choiceM},
$$
\int_{\Omega_{\rho}} \left(M^{p-1} -\mathcal{K}\right)a^{-\frac{1}{p-1}}(x)(u- \overline{u})^+dx \geq 0 
$$
and consequently $(u - \overline{u})^+=0$, i.e. $u \leq \overline{u}$ in $\Omega_{\rho}$. 

\medskip
Proceeding in a similar fashion for $-u$, we obtain that 
\begin{equation}\label{inequalityimport}
|u(x)| \leq \overline{u}(x)=  M\left(1 + \|u\|_{L^{\infty}(\Omega)}\right)\psi(x) \quad \hbox{for} \quad x\in \Omega_{\rho}.
\end{equation}

Inequality  \eqref{behavioru(x)} follows from \eqref{bdryestimate} and \eqref{inequalityimport}. This would complete the proof of the proposition.

\medskip

Next, we prove the claim. First notice that $u-\overline{u}\in W^{1,p}_0(\Omega,a)$ and Lemma \ref{positivenegativeparts} yields that $(u -\overline{u})^+ \in W^{1,p}_0(\Omega,a)$. 

\medskip
Fix $\upsigma\in (0,\rho)$ and let $\chi :\Omega \to [0,\infty)$ be a non-negative smooth function such that 
$$
\chi(x) = \left\{ 
\begin{aligned}
0&, \quad \hbox{for} \quad x\in \Omega - \Omega_{\upsigma},\\
1&, \quad \hbox{for} \quad x\in  \Omega_{\frac{\upsigma}{2}}.
\end{aligned}
\right.
$$

Since
$$
(u-\overline{u})^+ := \chi (u-\overline{u})^+ + (1-\chi)(u-\overline{u})^+ \quad \hbox{in}\quad \Omega,  
$$
it suffices to prove that $\chi(u-\overline{u})^+$ and $(1-\chi)(u-\overline{u})^+$ belong to $W^{1,p}_0(\Omega_{\rho},a)$. We prove these two statements as follows.

\medskip
First, approximating $(u -\overline{u})^+$ in $W^{1,p}_0(\Omega,a)$ with functions in $C^{\infty}_c(\Omega)$ yields $\chi(u-\overline{u})^+ \in W^{1,p}_0(\Omega_{\rho},a)$.

\medskip
On the other hand, from Proposition \ref{SerrinRegul}, $u\in C(\Omega)$ and thus $(u-\overline{u})^+ \in C(\overline{\Omega}_{\rho}-\pp \Omega)$. Using \eqref{distweight}, \eqref{bdryestimate}
and  \eqref{choiceM}, we find that $\overline{u} \geq u$ on {$\{{\rm dist}(\cdot,\pp \Omega)=\rho\}$} and hence 
\begin{equation}\label{boundarurho1cond}
(u-\overline{u})^+ =0\quad  \hbox{on} \quad {\{{\rm dist}(\cdot,\pp \Omega)=\rho\}}.
\end{equation}

\medskip
 
Consequently, $(1-\chi) (u-\overline{u})^+\in W^{1,p}(\Omega_{\rho},a)\cap C(\overline{\Omega}_{\rho})$ and from \eqref{boundarurho1cond}, 
$$
(1-\chi)(u-\overline{u})^+|_{\pp \Omega_{\rho}}=0.
$$ 

Lemma \ref{approxbyC1uptobdry} implies that $(1-\chi)(u-\overline{u})^+ \in W^{1,p}_0(\Omega_{\rho},a)$. This proves the claim and also completes the proof of the proposition.
\end{proof}

%\bigskip
%%%%%%%%%%%%%%%%%%%%%%%%%%%%%%%%%%%%%%%%
%%%%%%%%%%%%%%%%%%%%%%%%%%%%%%%%%%%%%%%%
%   multiplicity result and the proof of                   %            Theorem \ref{theo2}
%%%%%%%%%%%%%%%%%%%%%%%%%%%%%%%%%%%%%%%%
%%%%%%%%%%%%%%%%%%%%%%%%%%%%%%%%%%%%%%%%

\section{Proof of Theorem \ref{multiplicitytheorem}}\label{degreemultiplicity}

In this part we follow the scheme from Chapter 5 in \cite{DRABEKMILOTA} to prove Theorem \ref{multiplicitytheorem}. Also, we make use of the notations and conventions introduced in Sections \ref{introduction} and  \ref{preliminaries}.

\medskip

Assume hypotheses {\it (A1)-(A2)} and {\it (B1)-(B2)}. Using {\it (B2)}, let $\rho_1\in (0,\rho_0)$ be such that 
\begin{equation}\label{reductionH4}
\forall x\in \Omega_{\rho_1}:\quad c_1 a^{-\frac{1}{p-1}}(x) \leq b(x) \leq c_2 a^{-\frac{1}{p-1}}(x),
\end{equation}
where $\Omega_{\rho_1}$ is as in \eqref{def:Omegarho}.

\medskip
Consider the space $C(\overline{\Omega})$ endowed with the norm $\|\cdot\|_{L^{\infty}(\Omega)}$ and its subspace
$$
C_0(\overline{\Omega}):=\{w \in  C(\overline{\Omega})\,:\, w=0 \quad \hbox{on} \quad \pp \Omega\}.
$$

Next, consider the space 
$$
Y_a:=\left\{\tilde{b} \in L^{\frac{q}{q-p}}(\Omega)\,:\, a^{\frac{1}{p-1}}|\tilde{b}| \in L^{\infty}(\Omega_{\rho_1})\right\}
$$
where $q \in [p,p_s^*)$ is described in hypothesis {\it (B1)} and $Y_{a}$ is endowed with the norm
$$
\|\tilde{b}\|_{Y_a}:=\|\tilde{b}\|_{L^{\frac{q}{q-p}}(\Omega)} + \|a^{\frac{1}{p-1}}\tilde{b}\|_{L^{\infty}(\Omega_{\rho_1})}.
$$

Recall also that $\frac{q}{q-p}:=\infty$ if $q=p$.

\medskip
The following lemma will help us to set up the functional analytic settings to o
{carry out} the proof of Theorem \ref{multiplicitytheorem}.

\begin{lemma}\label{supersolution}
Let $r\in \big[\frac{p_s^*}{p_s^*-1},\infty\big]$ and $\tilde{b}\in L^{r}(\Omega)$. The BVP
 \begin{equation}\label{equation4}
-\rdiv\left(a(x)|\nabla w|^{p-2}\nabla w\right) =  \tilde{b}(x) \quad  \hbox{in} \quad \,\, \Omega, \qquad w=0 \quad \hbox{on} \quad \pp\Omega,
\end{equation}
has a unique solution $w \in W^{1,p}_0(\Omega,a)$, which satisfies the estimate
\begin{equation}\label{est:resolvent1}
\|w\|_{W^{1,p}_0(\Omega,a)}\leq C\|\tilde{b}\|^{\frac{1}{p-1}}_{L^{r}(\Omega)}
\end{equation}
for some constant $C>0$, depending only on $p,s,r$, $\Omega$ and the function $a$.

\medskip
If in addition $\tilde{b}\in Y_{a}$, then the solution $w$ of \eqref{equation4} belongs also to $ C_0(\overline{\Omega})$, it is locally H\"{o}lder continuous in $\Omega$ and satisfies the estimate
\begin{equation}\label{est:resolvent2}
%\|w\|_{W^{1,p}_0(\Omega,a)}+
\|w\|_{L^{\infty}(\Omega)}+\|{\rm d}^{-1}w\|_{L^{\infty}(\Omega_{\rho_1})}\leq {C}\,\big(1 + \|\tilde{b}\|_{L^{\frac{q}{q-p}}(\Omega)}+\|\tilde{b}\|^{\frac{1}{p-1}}_{L^{\frac{q}{q-p}}(\Omega)}\big)
\end{equation}
for some constant $C>0$ depending only on $p,q,s$,$\Omega$ and the function $a$.
\end{lemma}

\begin{proof} Given  $r\in [\frac{p_s^*}{p_s^*-1},\infty]$ and $\tilde{b}\in L^{r}(\Omega)$, 
the existence of a unique solution $w \in W^{1,p}_0(\Omega,a)$ of \eqref{equation4}, follows readily by a minimization procedure and a weak lower semicontinuity argument or by the use of theory of monotone operators.

\medskip
Set $\frac{r}{r-1}:=1$, if $r=\infty$. The weak form of \eqref{equation4} (see \eqref{weakaslnMODELDEGQUASILINEAR}), the fact that $\frac{r}{r-1}\in [1,p_s^*]$ together with H\"{o}lder's inequality and the Sobolev embedding in \eqref{WeightedSobolevEmbedding}, imply that
$$
\|w\|^p_{W^{1,p}_0(\Omega,a)}=\int_{\Omega}\tilde{b}w dx \leq \|\tilde{b}\|_{L^{r}(\Omega)}\|w\|_{L^{\frac{r}{r-1}}(\Omega)}\leq \tilde{C}\|\tilde{b}\|_{L^{r}(\Omega)}\|w\|_{W^{1,p}_0(\Omega,a)}
$$
for some constant $\tilde{C}>0$ depending only on $p,s,r$, $\Omega$ and $a$. The estimate in \eqref{est:resolvent1} follows taking $C=\tilde{C}^{\frac{1}{p-1}}$.

\medskip
Assume now that $\tilde{b}\in Y_a$.  Propositions \ref{BrezisKato} and \ref{SerrinRegul} imply, respectively, that $w\in L^{\infty}(\Omega)$ and $w$ is locally H\"{o}lder continuous in $\Omega$. Also,  the estimate \eqref{BrezisKatoLINFTY} in Proposition \ref{BrezisKato} yields that 
$$
\|w\|_{L^{\infty}(\Omega)} \leq C(1 + \|\tilde{b}\|_{L^{1}(\Omega)} + \|w\|_{L^{p}(\Omega)})
$$
and from the Friedrich's inequality, the estimate in \eqref{est:resolvent1} with $r=\frac{q}{q-p}$ and H\"{o}lder's inequali\-ty we find that 
\begin{equation}\label{est:resolvent3}
\|w\|_{L^{\infty}(\Omega)} \leq C\Bigl(1 + \|\tilde{b}\|_{L^{\frac{q}{q-p}}(\Omega)} + \|\tilde{b}\|^{\frac{1}{p-1}}_{L^{\frac{q}{q-p}}(\Omega)}\Bigr).
\end{equation}

Next, we apply Proposition \ref{BDCONDITION}, with $\rho=\rho_1$, to find that $w\in C_{0}(\overline{\Omega})$. The estimate in \eqref{est:resolvent2} follows from the estimate  \eqref{est:resolvent3} and the estimate \eqref{behavioru(x)} in Proposition \ref{BDCONDITION}, again with $\rho=\rho_1$. This completes the proof of the lemma.
\end{proof}

Using Proposition \ref{supersolution}, let $\LL:L^{\frac{p_s^*}{p_s^*-1}}(\Omega) \to W^{1,p}_0(\Omega,a)$ denote the resolvent operator of \eqref{equation4} and observe that $\LL:Y_a \to C_0(\overline{\Omega})$ is well defined.

\begin{lemma}\label{resolventoperator}The operator $\LL:Y_{a}\to C_0(\overline{\Omega})$ is compact and continuous.
\end{lemma}

\begin{proof} {\it Step 1:  compactness.} Assume $\{\tilde{b}_n\}_n \subset Y_a$ is bounded. Thus, there exists $\mathcal{K}>0$ such that
\begin{equation}\label{boundedweight}
\forall n\in \Bbb N:\quad \|\tilde{b}_n\|_{L^{\frac{q}{q-p}}(\Omega)} + \|a^{\frac{1}{p-1}}\tilde{b}_n\|_{L^{\infty}(\Omega_{\rho_1})}\leq \mathcal{K}.
\end{equation}

Noticing that $ L^{\frac{q}{q-p}}(\Omega)\subset L^{\frac{p_s^*}{p_s^*-1}}(\Omega)$ and using the reflexivity and separability of the space $L^{\frac{p_s^*}{p_s^*-1}}(\Omega)$ we may assume, passing to a subsequence if necessary,  the existence of $\tilde{b}_0 \in L^{\frac{p_s^*}{p_s^*-1}}(\Omega)$ such that for every $\varphi \in W^{1,p}_0(\Omega,a)$,
\begin{equation}\label{weakconvbn}
\lim \limits_{n\to \infty}\int_{\Omega} \tilde{b}_n(x)\varphi dx = \int_{\Omega} \tilde{b}_0(x)\varphi dx.
\end{equation}

Next, for every $n\in \mathbb{N}\cup \{0\}$, set 
\begin{equation}\label{functionaleqn}
w_n:= \LL(\tilde{b}_n)
\end{equation} 
and observe from \eqref{est:resolvent1}, \eqref{boundedweight} and \eqref{functionaleqn} that $\{w_n\}_{n\in \mathbb{N}\cup \{0\}}$ is bounded in $W^{1,p}_0(\Omega,a)$.

\medskip
Proceeding as in the Proof of Theorem 1.1 in \cite{AGUDELODRABEK1} (p. 1163-1164) using the embeddings in \eqref{embeddingimportant} and the reflexivity and separability of the space $W^{1,p} _0(\Omega,a)$, we find $\tilde{w}_0\in W^{1,p}_{0}(\Omega,a)$ such that up to a subsequence,
\begin{equation}\label{integralid}
\lim_{n\to \infty}\int_{\Omega}a(x)|\nabla w_n|^{p-2}\nabla w_n\cdot \nabla \varphi dx = \int_{\Omega}a(x)|\nabla \tilde{w}_0|^{p-2}\nabla \tilde{w}_0\cdot \nabla \varphi dx, \quad \forall \, \varphi \in W^{1,p}_0(\Omega,a).
\end{equation}

We can assume also that $w_n(x) \to \tilde{w}_0(x)$ a.e. $x \in \Omega$.

\medskip
Putting together \eqref{weakconvbn}, \eqref{functionaleqn} and  \eqref{integralid},
 we obtain that
\begin{equation}\label{eqntildew2}
\int_{\Omega} \left(a(x)|\nabla \tilde{w}_0|^{p-2}\nabla \tilde{w}_0 \cdot \nabla \varphi - \tilde{b}_0(x)\cdot \varphi \right)dx =0, \quad \forall \varphi \in W^{1,p}_0(\Omega,a).
\end{equation}

From \eqref{eqntildew2} and Lemma \ref{supersolution}, $\tilde{w}_0=\LL(\tilde{b}_0)=w_0 \in W^{1,p}_0(\Omega,a)$. 

\medskip
The definition of $w_n$  in \eqref{functionaleqn}, together with the estimates in \eqref{est:resolvent2} and \eqref{boundedweight} yield that the sequence $\{w_n\}_n$ is bounded in $L^{\infty}(\Omega)$.

\medskip
Next, we prove the existence of a subsequence of $\{w_n\}_n$, say $\{w_{n_k}\}_k$, such that 
\begin{equation}\label{morethancompactness}
w_{n_k}=\LL(\tilde{b}_{n_k})\to w_0=\LL(\tilde{b}_0) \quad \hbox{strongly in} \quad L^{\infty}(\Omega).
\end{equation}

Proposition \ref{SerrinRegul} applied to each domain in the sequence of open sets
$$
D_k:=\left\{x\in \Omega\,:\,\rdist(x,\pp \Omega)>\frac{1}{k}\right\} \quad \hbox{for} \quad k \in \mathbb{N},
$$
and performing a domain approximation procedure (see the proof of Theorem 1.1 in \cite{AGUDELODRABEK1} p. 1164-1166) and a diagonal argument, we pass to a subsequence, denoted again by $\{w_n\}_n$,  such that $w_n\to {w}_0$ uniformly over compact sets of $\Omega$.

\medskip
Using again the estimates in  \eqref{est:resolvent2} and \eqref{boundedweight}, we find a constant $c>0$ such that for every $n\in \mathbb{N}$ and every $x\in \Omega_{\rho_1}$,
\begin{equation*}\label{xx}
|w_n(x)|\big) \leq c \,{\rm d}(x).
\end{equation*}

Since $w_n(x)\to w_0(x)$ for a.e. $x\in \Omega$, we also find that for a.e. $x\in \Omega_{\rho_1}$,
\begin{equation*}\label{xx}
|{w}_0(x)|\leq c \,{\rm d}(x).
\end{equation*}

We conclude that  
\begin{equation}\label{xxx}
\|w_n\|_{L^{\infty}(\Omega_{\rho})} + \|{w}_0\|_{L^{\infty}(\Omega_{\rho})}\to  0, \quad \hbox{as }\quad \rho\to 0
\end{equation}
uniformly in $n\in \mathbb{N}$, where $\Omega_{\rho}$ is defined in \eqref{def:Omegarho}.

\medskip
The fact that $w_n \to {w}_0$ uniformly over compact subsets of $\Omega$ and \eqref{xxx} imply that $w_n \to {w}_0$ strongly in $L^{\infty}(\Omega)$. This completes the proof of the Step 1.

\medskip
{\it Step 2: continuity.} We proceed as in {\it Step 1} with only slight changes. The details are left to the reader. This completes the proof of the lemma. 
\end{proof}

Next, let $F=F(x,\zeta)$ be a Caratheodory function and let $\underline{u}, \overline{U}\in W^{1,p}(\Omega,a)\cap C(\overline{\Omega})$ be a subsolution and a supersolution to \eqref{MODELDEGQUASILINEAR}, respectively. Assume further that $\underline{u} \leq \overline{U}$ in $\overline{\Omega}$.

\medskip
Consider the truncated BVP
\begin{equation}\label{MODELDEGQUASILINEARTRUNC}
-\rdiv\left(a(x)|\nabla u|^{p-2}\nabla u\right)=  \tilde{F}(x,u) \quad  \hbox{in}  \quad \,\,\Omega,\qquad
u=0 \quad  \hbox{on}  \quad  \pp \Omega,
\end{equation}
where
$$
\widetilde{F}(x,\zeta):=\left\{
\begin{aligned}
&F(x,\zeta),& \quad \underline{u}(x)\leq &\zeta\leq \overline{U}(x)&\\
&F(x,\underline{u}(x)), & \quad &\zeta\leq \underline{u}(x)&\\
&F(x,\overline{U}(x)),& \quad  &\overline{U}(x)\leq \zeta&
\end{aligned}
\right. \quad \hbox{for}\quad x\in \Omega, \quad \zeta\in\R.
$$

Observe that $\tilde{F}$ is also a Caratheodory function.

\begin{lemma}\label{Slnstruncated=Slnsorig}
Let $u\in W^{1,p}(\Omega,a)\cap C_0(\overline{\Omega})$. If $u$ is a solution of \eqref{MODELDEGQUASILINEARTRUNC}, then $u$ is also a solution of \eqref{MODELDEGQUASILINEAR} with $\underline{u}\leq u \leq \overline{U}$ in $\Omega$.
\end{lemma}

\begin{proof}
Let $u\in W^{1,p}(\Omega,a)\cap C_0(\overline{\Omega})$ be a solution of \eqref{MODELDEGQUASILINEARTRUNC}. Since $\underline{u}\leq 0$ on $\pp \Omega$, $(\underline{u} - u)^+$ belongs to $W^{1,p}(\Omega,a)\cap C_0(\overline{\Omega})$. Lemma \ref{positivenegativeparts} implies that $(\underline{u}-u)^+\in W^{1,p}_0(\Omega,a)$.

\medskip
From \eqref{weakaslnMODELDEGQUASILINEAR}, with $\tilde{F}(x,\zeta)$ instead of $F(x,\zeta)$, and \eqref{weakaslnMODELDEGQUASILINEAR2},
$$
\begin{aligned}
0 \leq & \int_{\Omega}a(x)\left(|\nabla\underline{u}|^{p-2}\nabla \underline{u} -|\nabla u|^{p-2}\nabla u\right)\cdot \nabla (\underline{u} - u)^+ dx \\
\leq &  \int_{\{ \underline{u} \geq u \}}(F(x,\underline{u})-\widetilde{F}(x,u))(\underline{u}-u)^+ dx =0
\end{aligned}
$$
{from where we obtain $\underline{u}\leq u$ in a.e. $\Omega$. Proceeding in similar fashion we can verify that $\overline{U} \geq u$.} This proves the lemma.

\end{proof}

Let $F$ satisfy  \eqref{growthF(x,u)0}. Set 
$$
M:= \max\left(\|\underline{u}\|_{L^{\infty}(\Omega)}, \|\overline{U}\|_{L^{\infty}(\Omega)}\right) +1
$$
so that there exists $C_{M}>0$ such that for a.e. $x\in \Omega$ and every $\zeta \in \R$,
\begin{equation}\label{growthF(x,u)0strongertilde}
|\tilde{F}(x,\zeta)|\leq C_M b(x).
\end{equation}

Thanks to \eqref{reductionH4} and  \eqref{growthF(x,u)0strongertilde} $\tilde{F}(x,\zeta)$ satisfies \eqref{growthF(x,u)0stronger} and for any $u\in C_0(\overline{\Omega})$, $\tilde{F}(\cdot,u(\cdot))\in Y_a$.

\medskip
From Lemma \ref{supersolution}, the operator 
\begin{equation*}\label{fixedpointeqn}
T:C_0(\overline{\Omega})\to C_0(\overline{\Omega}), \qquad T(u):= \LL\left(\tilde{F}\big(\cdot,u(\cdot)\big)\right), \qquad u\in C_0(\overline{\Omega})
\end{equation*}
is well defined.

%\medskip
\begin{proof}[Proof of Theorem \ref{multiplicitytheorem}] 

{\it Step 1: compactness and continuity of $T$.}  The compactness of $T$ is a direct consequence of \eqref{growthF(x,u)0strongertilde} and Lemma \ref{resolventoperator}.

\medskip
To prove the continuity of $T$ we proceed as follows. Let $\{u_n\}_{n\in \Bbb N \cup \{0\}}\subset C_0(\overline{\Omega})$ such that $u_n \to u_0$ uniformly in ${\Omega}$. In particular,
$$
\forall x\in \Omega
:\quad u_n(x)\to u_0(x), \quad \hbox{as} \quad n\to \infty. 
$$

Since $\tilde{F}$ is a Caratheodory function, 
\begin{equation}\label{pointwiseconvergenceI}
\forall x\in \Omega
:\quad \tilde{F}(x,u_n(x))\to \tilde{F}(x,u_0(x)), \quad \hbox{as} \quad n\to \infty. 
\end{equation}

For every $n\in \Bbb N\cup \{0\}$, set
$$
\tilde{b}_n:=\tilde{F}(\cdot,u_n(\cdot))\quad  \hbox{and} \quad w_n:=\LL(\tilde{b}_n)=T(u_n).
$$ 

We remark that $\{\tilde{b}_n\}_{n\in \Bbb N}\subset Y_a$, but it may not converge in $Y_a$ and hence the continuity of $T$ does not follow directly from the continuity of $\LL$. 

\medskip
We overcome this difficulty as follows. Due to \eqref{growthF(x,u)0strongertilde}, \eqref{boundedweight} is satisfied. This fact, together with \eqref{pointwiseconvergenceI} and the Dominated convergence Theorem imply that \eqref{weakconvbn} holds true.
  
\medskip  
Proceeding in the same fashion as in Step 1 in the proof of Lemma \ref{resolventoperator}, we find that passing to a subsequence if necessary, \eqref{morethancompactness} holds also true as well, i.e.
\begin{equation*}\label{morethancompactnessI}
T(u_{n})\to T(u_0) \quad \hbox{strongly in} \quad L^{\infty}(\Omega).
\end{equation*}

This proves the continuity of $T$ and completes the proof of Step 1.

\medskip

{\it Step 2: existence.} Let $u\in C_0(\overline{\Omega})$ and set $w:=\tilde{T}(u)$, i.e., $w= \LL\Bigl(\tilde{F}\bigl(\cdot,u(\cdot)\bigr)\Bigr)$. Estimate \eqref{est:resolvent2} with $\tilde{b}=\tilde{F}\bigl(\cdot,u(\cdot)\bigr)$ and inequality \eqref{growthF(x,u)0strongertilde} imply the existence of a constant $\hat{C}=\hat{C}(\underline{u},\overline{U},p,q,a,b)>0$ such that
$$
\|w\|_{L^{\infty}(\Omega)}
\leq   \hat{C} \big(1 + \|b\|_{L^{\frac{q}{q-p}}(\Omega)}+\|b\|_{L^{\frac{q}{q-p}}(\Omega)}^{\frac{1}{p-1}}\big).
$$

Setting 
$$
R_0:= \hat{C} \big(1 + \|b\|_{L^{\frac{q}{q-p}}(\Omega)}+\|b\|_{L^{\frac{q}{q-p}}(\Omega)}^{\frac{1}{p-1}}\big)
$$
we conclude that for every $u\in C_0(\overline{\Omega})$, 
$$
\|{T}(u)\|_{L^{\infty}(\Omega)}\leq R_0.
$$

\medskip
Fix $R>R_0$ and let $B_R$ denote the open ball in $C_0(\overline{\Omega})$ of radius $R>0$ and centered at the origin and notice that $\tilde{T}(B_R) \subset B_R$. 

\medskip
{\it Schauder's Fixed Point Theorem} yields the existence of $u \in B_R$ such that $u= T(u)$. Lemma \ref{supersolution} implies that $u\in W^{1,p}_0(\Omega,a)\cap C(\overline{\Omega})$ and solves \eqref{MODELDEGQUASILINEARTRUNC}.

\medskip

Finally, Lemma \ref{Slnstruncated=Slnsorig} implies that $u$ also solves \eqref{MODELDEGQUASILINEAR} and this completes the proof of Theorem \ref{multiplicitytheorem}.
\end{proof}

%%%%%%%%%%%%%%%%%%%%%%%%%%%%%%%%%%%%%%%%
%%%%%%%%%%%%%%%%%%%%%%%%%%%%%%%%%%%%%%%%
% sub and super solution and proof of 
%              theorem 1
%%%%%%%%%%%%%%%%%%%%%%%%%%%%%%%%%%%%%%%%
%%%%%%%%%%%%%%%%%%%%%%%%%%%%%%%%%%%%%%%%

\section{Proof of theorem \ref{theo2}.}
In this final section we prove Theorem \ref{theo2}. Besides {\it (A1)-(A2)} and {\it (B1)-(B2)}, we assume also {\it (f1)-(f3)}. We adopt the notations and conventions from the previous sections.

\medskip
\begin{prop}\label{SUBSPERSLNS0}
For every $\la>0$, there exist $\overline{U}_{\la} \in W^{1,p}(\Omega,a)\cap C_0(\overline{\Omega})$ a supersolution to \eqref{DEGQUASILINEAR} with $\overline{U}_{\lambda}>0$ in $\Omega$. 
\end{prop}

\medskip

Although the proof of Proposition \ref{SUBSPERSLNS0} is rather standard, we include it here for the sake of completeness.

\medskip
\begin{proof} Let $\la>0$ be arbitrary. Set $\overline{U}=M\psi$, where $\psi$ is the solution to \eqref{equation4} with $\tilde{b}=b$ and $M>0$ is a parameter to be chosen later. Since $b(x)\geq 0$, from Harnack's inequality (see Theorem 1.9, \cite{DRABEKKUFNERNICOLOSI}), $\overline{U}$ is strictly positive inside $\Omega$ and from Lemma \ref{supersolution}, $\overline{U} \in W^{1,p}(\Omega,a) \cap C_0(\overline{\Omega})$.

\medskip
Consider the function
$$
\hat{f}({\zeta}):=\max_{{\rm z}\in[0,\zeta]}f({\rm z})
$$
which is monotone increasing, $\hat{f}\geq f$ and from {\it (f1)}, 
\begin{equation}\label{uppernonlinearity}
\limsup_{\zeta\to \infty}\frac{\hat{f}(\zeta)}{\zeta^{p-1}}=0.
\end{equation}

\medskip
Fix $\ep=\ep(\lambda)>0$ such that
\begin{equation}\label{superslninequality}
\ep \|\psi\|_{L^{\infty}(\Omega)}^{p-1}\leq \frac{1}{\la}.
\end{equation}

\medskip
Using \eqref{uppernonlinearity} we get the existence of $\zeta_{\lambda}>1$ such that for every $\zeta \geq \zeta_{\la}$,
$$
0\leq \hat{f}(\zeta)\leq \ep\zeta^{p-1}. 
$$

Choose $M=M(\la)$ satisfying 
\begin{equation}\label{Mlambda}
M\|\psi\|_{L^{\infty}} \geq \zeta_{\la},
\end{equation}
so  that
$$
\hat{f}\left(M \|\psi\|_{L^{\infty}(\Omega)}\right)\leq  \ep M^{p-1}\|\psi\|^{p-1}_{L^{\infty}(\Omega)}
$$
and from \eqref{superslninequality},
$$
\la \hat{f}(M \|\psi\|_{L^{\infty}(\Omega)})\leq M^{p-1}.
$$

\medskip
Then, formally, we compute
\begin{eqnarray*}
\la b(x)\hat{f}\left(M \|\psi\|_{L^{\infty}(\Omega)}\right) &\leq& M^{p-1} b(x)\\
&=&-\rdiv(a(x)|\nabla \overline{U}|^{p-2}\nabla \overline{U})
\end{eqnarray*}
so that $\overline{U}$ is a supersolution to \eqref{DEGQUASILINEAR}, proving the proposition.
\end{proof}

\begin{prop}\label{SUBSUPERSLNS}
Let $\Omega$ be a smooth bounded domain such that $\pp \Omega$ has non-negative mean curvature.
There exists $\lambda_{*}>0$ such that for every $\la\geq \la_{*}$, there exists $\underline{u}_{\la}\in W^{1,p}(\Omega,a)\cap C_0(\overline{\Omega})$ a subsolution to \eqref{DEGQUASILINEAR} with $\underline{u}_{\lambda}>0$ in $\Omega$.
\end{prop}

\begin{proof} First, we describe the general strategy. Let $r\in (0,p-1)$ be as in {\it (f2)}. Thanks to {\it (f1)}, {\it (f2)} and {\it (f3)}, we can also fix a monotone increasing function $\tilde{f}\in C[0,\infty)$ such that 
$$
\tilde{f}\leq f, \quad -\infty<\tilde{f}(0)< f(0)\leq 0
$$ 
and, moreover, satisfying that
\begin{equation}\label{fsubslnmodf}
\mu:=\lim_{\zeta \to \infty}\frac{\tilde{f}(\zeta)}{\zeta^{r}} \in (0,\infty).
\end{equation}

Thus, it suffices to find a subsolution for the BVP 
\begin{equation}\label{def:sublinearBVP}
-{\rm div}\bigl(a(x)|\nabla u|^{p-2}\nabla u\bigr)=\lambda b(x)\tilde{f}(u) \quad \hbox{in} \quad \Omega, \qquad u=0 \quad \hbox{on} \quad \partial \Omega. 
\end{equation}

To do so, we proceed in a series of steps. In what follows, we fix $\beta$ and  $\sigma$ such that 
\medskip 
\begin{equation}\label{exponentssubslntilde}
\frac{1}{p-1}<\sigma<\frac{1}{p-1-r} \quad \hbox{and}\quad 1<\beta<1+ \frac{1}{p-1}.
\end{equation}

{\it Step 1: profile near $\partial \Omega$.} Let $\rho_1\in (0,\rho_0)$ be as in \eqref{reductionH4}. Since $\pp \Omega$ has non-negative mean curvature, we may also assume that $\rho_1$ is such that
 \begin{equation}\label{positivemeancurvature}
\Lambda:=\min \limits_{\substack{\y \in \pp \Omega, \\ 0\leq y_N\leq \rho_1}}\pp_{y_N}\log\left(\sqrt{\det g}\right)\geq 0.
\end{equation}

Consider the decreasing function $A:\big[\frac{\rho_1}{2},\rho_1\big)\to (0,\infty)$ defined by
\begin{equation}\label{formulaAB}
A(\rho):=\beta^{p-1}\left[\int_0^{\rho}{\sf a}^{-\frac{1}{p-1}}(\tau)d\tau\right]^{-(p-1)}.
\end{equation}

We claim that there exists $\hat{\la}_1>0$ such that for every $\la \geq \hat{\la}_1$, there exists $\rho = \rho(\la)\in (\frac{\rho_1}{2},\rho_1)$ satisfying that
\begin{equation}\label{deratrhoI}
e^{\Lambda \rho}A(\rho)= \la^{1 - \sigma(p-1 - {r})}\int_{\rho}^{\rho_1}e^{-\Lambda \tau}{\sf a}^{-\frac{1}{p-1}}(\tau) d\tau.
\end{equation}

To prove the claim, consider the positive function 
$$
\left[\frac{\rho_1}{2},\rho_1\right)\ni \rho \mapsto {\rm t}(\rho):=\frac{\e^{\Lambda \rho}A(\rho)}{\int_{\rho}^{\rho_1}e^{\Lambda \tau}{\sf a}^{-\frac{1}{p-1}}(\tau) d\tau}
$$
so that ${\rm t}\in C\big[\frac{\rho_1}{2},\rho_1\big)$ and ${\rm t}(\rho)\to + \infty$ as $\rho \to \rho_1$. Let $\hat{\rho} \in \big[\frac{\rho_1}{2},\rho_1)$ be such that $
{\rm t}(\hat{\rho}):= \min \limits_{\rho \in \big[\frac{\rho_1}{2},\rho_1\big)} {\rm t}(\rho)$ and notice that ${\rm t}(\hat{\rho})>0$.

\medskip
Setting
\begin{equation*}\label{lambda*firstchoice}
\hat{\la}_1 := {\rm t}(\hat{\rho}) ^{\frac{1}{1 - \sigma(p-1-r)}},
\end{equation*}
we observe that the claim follows by a direct application of the {\it Intermediate Value Theorem}.

\medskip
Next, we consider the function $\rho:[\hat{\lambda}_1, \infty)\to [\frac{\rho_1}{2},\rho_1)$, where $\rho=\rho(\lambda)$ solves \eqref{deratrhoI}. 

\medskip

Next, let $\la \geq \hat{\la}_1$. Write $\rho=\rho(\la)$ and consider the function $\phi:[\rho,\rho_1]\to (0,\infty)$ defined by
\begin{equation}\label{formulaphi}
\phi(y_N):= 1 + \int_{\rho}^{y_N} e^{-\frac{\Lambda \zeta}{p-1}}{\sf a}^{-\frac{1}{p-1}}(\zeta)\left[e^{\Lambda \rho} A(\rho)- \la^{1 -\sigma(p-1-r)}\int_{\rho}^{\zeta}e^{\Lambda \tau}{\sf a}^{-\frac{1}{p-1}}(\tau)d\tau\right]^{\frac{1}{p-1}}d\zeta.
\end{equation}

Using \eqref{hyp:rma} we verify that $\phi$ is well defined, differentiable and ${\sf a}(y_N)|\phi|^{p-1}$ is absolutely continuous in $(\rho,\rho_1)$.

\medskip
On the other hand, \eqref{deratrhoI} yields that for any $y_N \in (\rho,\rho_1)$, 
\begin{equation}\label{deratrhoII}
e^{\Lambda \rho}A(\rho)> \la^{1 - \sigma(p-1 - {r})}\int_{\rho}^{y_N}e^{\Lambda \tau}{\sf a}^{-\frac{1}{p-1}}(\tau) d\tau.
\end{equation}

Thus, from \eqref{formulaphi} and \eqref{deratrhoII},
\begin{equation}\label{gluingfunct0}
\phi(\rho)=1\quad \hbox{and}\quad {\sf a}(y_N)|\pp_{y_N}\phi|^{p-2}\pp_{y_N}\phi> 0 \quad   \hbox{in} \quad (\rho,\rho_1).
\end{equation}

Directly from \eqref{formulaphi},
\begin{equation}\label{gluingfunct}
-\pp_{y_N}\left(e^{\Lambda y_N}{\sf a}(y_N)|\pp_{y_N}\phi|^{p-2}\pp_{y_N}\phi\right) = \la^{1 -\sigma(p-1-{r})}\,e^{\Lambda y_N}{\sf a}^{-\frac{1}{p-1}}(y_N) \quad \hbox{in} \quad (\rho,\rho_1)
\end{equation}
with the boundary conditions
\begin{equation}\label{gluingfunctbdcond}
{\sf a}(\rho)|\pp_{y_N}\phi(\rho)|^{p-2}\pp_{y_N}\phi(\rho)= A(\rho),  \quad {\sf a}(\rho_1)|\pp_{y_N}\phi(\rho_1)|^{p-2}\pp_{y_N}\phi(\rho_1)=0.
\end{equation}

\medskip

{\it Step 2: candidate for a subsolution of \eqref{def:sublinearBVP}.} Consider the decreasing function $B:\big[\frac{\rho_1}{2},\rho_1\big)\to (0,\infty)$ defined by
\begin{equation*}\label{formulaABI}
B(\rho):= \left[\int_{0}^{\rho}{\sf a}^{-\frac{1}{p-1}}(\zeta)d\zeta\right]^{-\beta}
\end{equation*}
and recall that for $\lambda \in [\hat{\lambda}_1,\infty)$, we write $\rho=\rho(\lambda)$. %Observe that $A(\rho(\la))$ and $B(\rho(\la))$ are uniformly bounded in $\la\in[\hat{\lambda}_1,\infty)$.

\medskip
Next, consider the function ${\rm v}:[0,\infty) \to \R$ defined by
\begin{equation}\label{definitrmv}
{\rm v}(y_N)=\left\{
\begin{array}{ccc}
B(\rho)\left(\int_0^{y_N} {\sf a}^{-\frac{1}{p-1}}(\zeta)d\zeta \right)^{\beta},&0 \leq y_N \leq \rho,\\ 
\phi(y_N),& \rho< y_N<\rho_1,\\
  \phi(\rho_1),&y_N\geq \rho_1.
 \end{array}
 \right.
\end{equation}

From \eqref{formulaphi}, \eqref{gluingfunctbdcond} and \eqref{definitrmv}, ${\rm v}\in C[0,\infty)\cap L^{\infty}(0,\infty)$ and ${\sf a}(y_N)|\pp_{y_N}{\rm v}|^{p-}\pp_{y_N}{\rm v}\in L^1(0,\infty)$ and it is absolutely continuous in $(0,\infty)$. 

\medskip
Consider the Fermi coordinates (see Section 2.3)
$$
(\y, y_N)\in \pp \Omega \times [0,\rho_1] \mapsto  x= \y+ \,y_N\,{\rm n}(\y)\in \overline{\Omega}_{\rho_1}
$$ 
with {\it induced metric} denoted by $g$ and associated {\it Jacobian determinant} $\sqrt{\det g}$. Recall that $y_N={\rm dist}(x,\pp \Omega)$.

Set 
\begin{equation}
\label{def:upsilonv}
\upsilon(x)=\left\{
\begin{array}{ccc}
{\rm v}({\rm dist}(x, \pp \Omega)) ,&x \in \overline{\Omega}_{\rho_1},\\
  {\rm v}(\rho_1),& x\in \Omega - \overline{\Omega}_{\rho_1}. 
 \end{array}
 \right.
\end{equation}

The properties of ${\rm v}$ imply that $\upsilon \in W^{1,p}(\Omega,a)\cap C_0(\overline{\Omega})$. 

\medskip
Omitting the explicit dependence on $\la$, set $\underline{u}=\la^{{\sigma}}\upsilon$ in $\Omega$. 
Abusing the notation, write $\underline{u}(x)=\underline{u}(y_N)$ in $\overline{\Omega}_{\rho_1}$. Next, we compute $a(x)|\nabla \underline{u}|^{p-2}\nabla \underline{u}$ in Fermi coordinates using Lemma \ref{pLapincoordinates}. From  \eqref{formulaAB}, \eqref{formulaphi} and \eqref{definitrmv}, 
\begin{equation}\label{derivubar}
\begin{small}
\begin{aligned}
{\sf a}(y_N)|\pp_{y_N}\underline{u}|^{p-2}\pp_{y_N}\underline{u}=&\la^{\sigma(p-1)}\,{\sf a}(y_N)\,|\pp_{y_N}\upsilon|^{p-2}\pp_{y_N}\upsilon
\\
=&\left\{
\begin{aligned}
\la^{{\sigma}(p-1)}B(\rho)^{p-1}{\beta}^{p-1}\left[\int_0^{y_N}{\sf a}^{-\frac{1}{p-1}}(\zeta)d\zeta\right]^{(\beta -1)(p-1)},&\quad  0<y_N \leq \rho,
\\
\la^{\sigma(p-1)}{\sf a}(y_N)|\pp_{y_N}\phi|^{p-2}\pp_{y_N}\phi,& \quad \rho < y_N< \rho_1.
\end{aligned}
\right.
\end{aligned}
\end{small}
\end{equation}

{From \eqref{gluingfunct}, \eqref{gluingfunctbdcond} and \eqref{derivubar},} we conclude that ${\sf a}(y_N)|\pp_{y_N}\underline{u}|^{p-2}\pp_{y_N}\underline{u}$ is absolutely continuous in $(0,\rho_1)$.

\medskip
{\it Step 3: our candidate is indeed a subsolution of \eqref{def:sublinearBVP}.} We now claim that there exists ${\lambda}_{*} \geq \hat{\la}_1$ such that for any $\la \geq {\la}_{*}$, $\underline{u}$ is a subsolution to \eqref{def:sublinearBVP}. To prove this claim, we proceed as follows.

\medskip
Let $\varphi\in C_c^{\infty}(\Omega)$, with $\varphi \geq 0$ in $\Omega$, be arbitrary. Abusing the notation, write in coordinates $\varphi(x)=\varphi(\y,y_N)$ in $\Omega_{\rho_1}$. 

\medskip
Using Lemma \ref{pLapincoordinates} and integrating by parts, we find that 
\begin{eqnarray*}
\int_{\Omega}a(x)|\nabla \underline{u}|^{p-2}\nabla \underline{u}\cdot \nabla \varphi dx&=&\int_{\pp \Omega}\int_{0}^{\rho_1}{\sf a}(y_N)|\pp_{y_N}\underline{u}|^{p-2}\pp_{y_N}\underline{u}\,\pp_{y_N}\varphi \, \sqrt{\det g}\, dy_N\, d\y\\
\\
&=&\underbrace{\int_{\pp \Omega}{\sf a}(y_N)[\pp_{y_N}\underline{u}(y_N)]^{p-1}\varphi(\y,y_N)\sqrt{\det g(\y,y_N)}\,\,\big{|^{y_N= \rho_1}_{y_N=0}}\, d\y}_{I} \quad \\
&& \underbrace{-\int_{\pp \Omega}\int_{0}^{\rho_1}\pp_{y_N} \left(\sqrt{{\rm det}g}\, {\sf a}(y_N) |\pp_{y_N} \underline{u}|^{p-2}\pp_{y_N}\underline{u}\right)\varphi \,dy_n\, d\y}_{II}.
\end{eqnarray*}

\medskip
From \eqref{gluingfunctbdcond} and since $\varphi \in C_c^{\infty}(\Omega)$, $I=0$ and therefore
\begin{equation}\label{estimaateI}
\int_{\Omega}a(x)|\nabla \underline{u}|^{p-2}\nabla \underline{u}\cdot \nabla \varphi dx =II.
\end{equation}

Next, we estimate $II$. Observe that
$$
\begin{aligned}
II=&\int_{\pp \Omega}\int_{0}^{\rho_1}-\frac{1}{\sqrt{{\rm det}g}}\pp_{y_N} \left(\sqrt{{\rm det}g}\, {\sf a}(y_N) |\pp_{y_N} \underline{u}|^{p-2}\pp_{y_N}\underline{u}\right)\varphi \, \sqrt{\det g}\, dy_N\, d\y\\
=& \underbrace{\int_{\pp \Omega}\int_{0}^{\rho}\cdots  \sqrt{\det g}\, dy_N\, d\y}_{II_1} + \underbrace{\int_{\pp \Omega}\int_{\rho}^{\rho_1}\cdots\, \sqrt{\det g}\, dy_N\, d\y}_{II_2},
\end{aligned}
$$
where we remark again that $\rho=\rho(\lambda)$ and $\rho(\lambda) \in [\frac{\rho_1}{2},\rho_1)$.

We estimate first $II_1$. The choices of $\sigma$ and $\beta$ in \eqref{exponentssubslntilde}, allow us to fix ${\la}_{*}\geq \hat{\la}_1$ such that
\begin{equation}\label{lambda**}
\underbrace{{\la_{*}^{1 - \sigma(p-1)}}c_2\bigl(-\tilde{f}(0)\bigr)}_{>0} \leq \beta^{p-1}B\bigl(\rho_1\bigr)^{(p-1)}(\beta -1)(p-1) \left[\int_0^{\rho_1}{\sf a}^{-\frac{1}{p-1}}(\zeta)d\zeta\right]^{(\beta -1)(p-1)-1},
\end{equation}
where $c_2>0$ is described in \eqref{reductionH4}.

\medskip

From \eqref{positivemeancurvature},\eqref{lambda**} and the fact that $B(\rho)$ is decreasing, we obtain for every $\la \geq {\la}_{*}$ and for every $y_N \in \bigl(0,\rho(\la)\bigr)$ that
%\begin{small}
\begin{multline}\label{almostthere}
-\la c_2 {\sf a}^{-\frac{1}{p-1}}(y_N)\tilde{f}(0)  \\
\leq 
\la^{\sigma(p-1)}\beta^{p-1}B(\rho)^{(p-1)} \left((\beta -1)(p-1)\left[\int_0^{y_N}{\sf a}^{-\frac{1}{p-1}}(\zeta)d\zeta\right]^{(\beta -1)(p-1)-1}{\sf a}^{-\frac{1}{p-1}}(y_N) \right. \\+
\left.\left[\int_0^{y_N}{\sf a}^{-\frac{1}{p-1}}(\zeta)d\zeta\right]^{(\beta -1)(p-1)}\Lambda \right).
\end{multline}
%\end{small}

Consequently, \eqref{positivemeancurvature} and \eqref{almostthere} yield that
$$
\begin{aligned}
-\la c_2 {\sf a}^{-\frac{1}{p-1}}(y_N)\tilde{f}(0)  
\leq & \la^{\sigma(p-1)}\beta^{p-1}B(\rho)^{(p-1)} \left((\beta -1)(p-1)\left[\int_0^{y_N}{\sf a}^{-\frac{1}{p-1}}(\zeta)d\zeta\right]^{(\beta -1)(p-1)-1}{\sf a}^{-\frac{1}{p-1}}(y_N) \right. \\
&\hspace{5.3cm}+\left.\left[\int_0^{y_N}{\sf a}^{-\frac{1}{p-1}}(\zeta)d\zeta\right]^{(\beta -1)(p-1)}\pp_{y_N}\log (\sqrt{{\rm \det} g}) \right)\\
= & \frac{\la^{\sigma(p-1)}\beta^{p-1}B(\rho)^{(p-1)}}{\sqrt{{\rm det} g}} \pp_{y_N}\left(\sqrt{{\rm det}g}\left[\int_0^{y_N}{\sf a}^{-\frac{1}{p-1}}(\zeta)d\zeta\right]^{(\beta -1)(p-1)}\right).
\end{aligned}
$$

From the previous inequality and \eqref{derivubar} we deduce that
\begin{equation}\label{subsolution01}
-\frac{1}{\sqrt{{\rm det}g}}\pp_{y_N} \left(\sqrt{{\rm det}g}\, {\sf a}(y_N) |\pp_{y_N} \underline{u}|^{p-2}\pp_{y_N}\underline{u}\right) \leq \la c_2{\sf a}^{-\frac{1}{p-1}}(y_N)\tilde{f}(0) \quad \hbox{in} \quad \pp\Omega \times (0,\rho).
\end{equation}

Using \eqref{subsolution01} and the fact that $\varphi  \geq 0$,
$$
\begin{aligned}
II_1= &\int_{\pp \Omega}\int_{0}^{\rho}-\frac{1}{\sqrt{{\rm det}g}}\pp_{y_N} \left(\sqrt{{\rm det}g}\, {\sf a}(y_N) |\pp_{y_N} \underline{u}|^{p-2}\pp_{y_N}\underline{u}\right)\varphi \, \sqrt{\det g}\, dy_N\, d\y\\
\leq & \int_{\pp \Omega} \int_0^{\rho}\la c_2 {\sf a}^{-\frac{1}{p-1}}(y_N)\tilde{f}(0) \varphi \, \sqrt{\det g}\, dy_N\, d\y.
\end{aligned}
$$

Pulling back the change of variables, 
\begin{equation}\label{subslnMAIN0}
II_1\leq  \int_{\Omega_{\rho}} \la c_2 a^{-\frac{1}{p-1}}(x)\tilde{f}(0)\varphi dx\leq 0.
\end{equation}

It then follows from \eqref{reductionH4} and the fact that $\tilde{f}$ is non-decreasing that
\begin{equation}\label{subslnMAIN1}
II_1\leq  \int_{\Omega_{\rho}} \la b(x)\tilde{f}(\underline{u})\varphi dx.
\end{equation}

Next, we estimate the integral $II_2$. Since $\underline{u}(y_N)=\lambda^{\sigma}\upsilon(y_N)=\lambda^{\sigma}\phi(y_N)$ for $y_N\in (\rho,\rho_1)$, we have
$$
\begin{aligned}
II_2=&\la^{\sigma(p-1)}\int_{\pp \Omega}\int_{\rho}^{\rho_1}-\frac{1}{\sqrt{{\rm det}g}}\pp_{y_N} \left(\sqrt{{\rm det}g}\, {\sf a}(y_N) |\pp_{y_N} \phi|^{p-2}\pp_{y_N}\phi\right)\varphi \, \sqrt{\det g}\, dy_N\, d\y.
\end{aligned}
$$

From \eqref{gluingfunct0}, the function $\phi(y_N)$, $y_N\in [\rho,\rho_1)$, is strictly increasing, bounded and $\phi(\rho)=1$. This fact and \eqref{gluingfunct} imply that \begin{equation*}\label{differentialinequality}
\begin{aligned}
-\frac{1}{\sqrt{\det g}}\pp_{y_N}\left(\sqrt{\det g}\,{\sf a}(y_N)\, \,|\pp_{y_N}\phi|^{p-2}\pp_{y_N}\phi\right) &\leq &  \la^{1 -\sigma(p-1-{r})}{\sf a}^{-\frac{1}{p-1}}(y_N)\\
&\leq& \la^{1 -\sigma(p-1-{r})}{\sf a}^{-\frac{1}{p-1}}(y_N) \phi^{{r}} 
\end{aligned}
\end{equation*}
in $(\rho,\rho_1)$. Therefore,
$$
II_2 \leq \la^{1 - \sigma(p-1)+ \sigma{r}}\int_{\pp \Omega}\int_{\rho}^{\rho_1} {\sf a}^{-\frac{1}{p-1}}(y_N) \phi^{{r}}  \varphi \, \sqrt{\det g}\, dy_N\, d\y.
$$

Pulling back the change of variables and using \eqref{reductionH4} and the definition of $\upsilon$ in \eqref{def:upsilonv}, we find that
\begin{equation}\label{inequalityphi1}
\begin{aligned}
II_2 \leq & \frac{\la^{1- \sigma(p-1) }}{c_1}\int_{\Omega_{\rho_1} - \overline{\Omega}_{\rho}}b(x) \la^{\sigma{r}}\phi^{{r}}  \varphi dx\\
\leq &\frac{\la^{1- \sigma(p-1) }}{c_1}\int_{\Omega - \overline{\Omega}_{\rho}}b(x) \la^{\sigma{r}}\upsilon^{{r}}  \varphi dx.
\end{aligned}
\end{equation}

We finish estimating the integral $II_2$ as follows. Fix $\ep \in (0,\mu)$, {where $\mu>0$ is as in \eqref{fsubslnmodf}}. Taking ${\la}_{*}$ larger if necessary, we may assume that 
\begin{equation}\label{lambda***}
\frac{{\la}_{*}^{-\sigma(p-1)}}{c_1(\mu-\ep)} \leq 1.
\end{equation}

Moreover, thanks to \eqref{fsubslnmodf} and the fact that $\upsilon \geq 1$ in $\Omega_{\rho_1}-\overline{\Omega}_{\rho}$, we may also assume that for every $\la \geq {\la}_{*}$, 
\begin{equation}\label{growthineq}
(\mu-\ep)\la^{\sigma {r}} \upsilon^{{r}} \leq \tilde{f}(\la^{\sigma} \upsilon) \quad \hbox{in} \quad \Omega_{\rho_1} -\overline{\Omega}_{\rho}.
\end{equation}

From \eqref{inequalityphi1}, \eqref{lambda***} and \eqref{growthineq},
\begin{equation}\label{subslnMAIN2}
\begin{aligned}
II_2 \leq& \int_{\Omega - \overline{\Omega}_{\rho}} \la^{1-\sigma(p-1)} b(x) \la^{\sigma r}\upsilon^r \varphi dx\\
\leq &\frac{\la^{1-\sigma(p-1)}}{c_1(\mu -\ep)}\int_{\Omega - \overline{\Omega}_{\rho}}b(x) \tilde{f}(\la^{\sigma}\upsilon) \varphi dx\\
\leq & \la\int_{\Omega - \overline{\Omega}_{\rho}}b(x) \tilde{f}(\underline{u})  \varphi dx
\end{aligned}
\end{equation}
for every $\lambda \geq {\la}_{*}$.

\medskip
Finally, putting together \eqref{subslnMAIN1} and  \eqref{subslnMAIN2}, for every $\la \geq {\la}_{*}$,
$$
\int_{\Omega}a(x)|\nabla \underline{u}|^{p-2}\nabla \underline{u} \cdot \nabla \varphi dx \leq \int_{\Omega}\la b(x)\tilde{f}(\underline{u})\varphi dx.
$$

Since $\varphi \in C^{\infty}_c(\Omega)$ with $\varphi \geq 0$ in $\Omega$ is arbitrary, we conclude that for any $\la \geq \la_{*}$, $\underline{u} \in W^{1,p}(\Omega,a)\cap C_0(\overline{\Omega})$ is a subsolution of \eqref{DEGQUASILINEAR}. This complete the proof of the proposition.
\end{proof}

\medskip
\begin{prop}\label{ordersubsuperslns}
Under the same assumptions as in Proposition \ref{SUBSUPERSLNS}, there exists $\la_0>0$ such that for every $\la \geq \la_0$, there exist $\underline{u}_{\la},\overline{U}_{\la}$ a positive subsolution and a positive super solution to \eqref{DEGQUASILINEAR} in $\overline{\Omega}$ with $\underline{u}_{\la} \leq \overline{U}_{\la}$ in $\overline{\Omega}$.
\end{prop}

\begin{proof}
We track down the proofs of Propositions \ref{SUBSPERSLNS0} and \ref{SUBSUPERSLNS}. Let $\la_{*}>0$ be as in the proof of Proposition \ref{SUBSUPERSLNS} and let $M=M(\lambda)$, $\lambda \geq \lambda_*$, be as in \eqref{Mlambda}. Let also $\underline{u}_{\la},\overline{U}_{\la}$ be as in  Propositions \ref{SUBSPERSLNS0} and \ref{SUBSUPERSLNS}, respectively.

\medskip
Let $\varphi \in W^{1,p}_0(\Omega,a)\cap C(\overline{\Omega})$ with $\varphi \geq 0$ be arbitrary. From the definition of the function $\upsilon$ in \eqref{def:upsilonv} and the estimates in \eqref{estimaateI}, \eqref{subslnMAIN0} and  \eqref{inequalityphi1},
we find that
$$
\begin{aligned}
\int_{\Omega} a(x)|\nabla \underline{u}_{\la}|^{p-2}\nabla \underline{u}_{\la} \cdot \nabla \varphi dx \leq & \int_{\Omega_{\rho}} \la c_2 a^{-\frac{1}{p-1}}(x)\tilde{f}(0)\varphi dx \\
& \hspace{1.3cm}+ \int_{\Omega-\Omega_{\rho}} \frac{\la^{1 - \sigma(p-1)+ \sigma r}}{c_1}b(x)\upsilon^r \varphi dx
\end{aligned}
$$

Since $\tilde{f}(0)\leq 0$, we conclude that
$$
\int_{\Omega} a(x)|\nabla \underline{u}_{\la}|^{p-2}\nabla \underline{u}_{\la} \cdot \nabla \varphi dx \leq \int_{\Omega} \frac{\la^{1 - \sigma(p-1)+ \sigma r}}{c_1}b(x)\upsilon^r \varphi dx
$$
for every $\varphi \in C^{\infty}_c(\Omega)$ with $\varphi  \geq 0$ in $\Omega$.

\medskip
Therefore, 
$$
\begin{aligned}
0 \leq & \int_{\Omega}a(x)\left(|\nabla\underline{u}_{\la}|^{p-2}\nabla \underline{u}_{\la} -|\nabla \overline{U}_{\la}|^{p-2}\nabla \overline{U}_{\la}\right)\cdot \nabla (\underline{u}_{\la} - \overline{U}_{\la})^+ dx \\
\leq &  \int_{\{ \underline{u}_{\la} \geq \overline{U}_{\la} \}}b(x)\left(\frac{\la^{1-\sigma(p-1-{r})}}{c_1} \upsilon^{{r}}(x)-M^{p-1}\right)(\underline{u}_{\la}-\overline{U}_{\la})^+ dx.
\end{aligned}
$$

\medskip
Next, observe from \eqref{gluingfunct0} that the function ${\rm v}(y_N)$, $y_N\in [0,\infty)$, defined in\eqref{definitrmv}  is strictly increasing and bounded with $\|{\rm v}\|_{L^{\infty}[0,\infty)}=\phi(\rho_1)$. The latter fact together with the definition of $\upsilon$ in \eqref{def:upsilonv} imply that $\|\upsilon\|_{L^{\infty}(\Omega)}= \phi(\rho_1)$.

\medskip
We may take $\la_0\geq \la_{*}$ and for any $\la \geq \la_0$ we may choose $M=M(\la)$ larger if necessary, satisfying \eqref{Mlambda} and
$$
c_1 \la^{\sigma (p-1-{r})-1}M^{p-1}(\la)\geq \phi^r(\rho_1)=\|\upsilon\|^r_{L^{\infty}(\Omega)}.
$$ 

Consequently,
$$
\int_{\{ \underline{u}_{\la} \geq \overline{U}_{\la} \}}b(x)\left(\frac{\la^{1-\sigma(p-1-{r})}}{c_1} \upsilon^{{r}}(x)-M^{p-1}\right)(\underline{u}_{\la}-\overline{U}_{\la})^+ dx \leq 0
$$
and thus $(\underline{u}_{\la} - \overline{U}_{\la})^+=0$, i.e., $\underline{u}_{\la}\leq \overline{U}_{\la}$ in $\Omega$. This completes the proof of the proposition.
\end{proof}

\begin{proof}[Proof Theorem \ref{theo2}]
Let $\la_0$ be as in Proposition \ref{ordersubsuperslns} and let $\la \geq \la_0$ be arbitrary, but fixed. Omitting the dependence on $\la$, let $\overline{U}, \underline{u}$ be the ordered subsolution and supersolution to \eqref{DEGQUASILINEAR}, respectively, predicted in Propositions \ref{SUBSPERSLNS0}, \ref{SUBSUPERSLNS} and 
\ref{ordersubsuperslns}.

\medskip
Set $\zeta_0:= \max (\|\underline{u}\|_{L^{\infty}(\Omega)}, \|\overline{U}\|_{L^{\infty}(\Omega)})+1$ and consider the function 
\begin{equation}\label{defFTheo2}
F(x,\zeta):= 
\left\{\begin{aligned}
\la b(x)f(\zeta),& \quad |\zeta|\leq \zeta_0\\
\la b(x)f(\zeta_0),& \quad |\zeta|>\zeta_0. 
\end{aligned}
\right.
\end{equation}

With this definition of $F$, it is clear that \eqref{growthF(x,u)0} is satisfied and $\underline{u}, \overline{U}$ are a sub- and a supersolution of \eqref{MODELDEGQUASILINEAR}, respectively. 

\medskip
A direct application of Theorem \ref{multiplicitytheorem} yields the existence of $u\in W^{1,p}_0(\Omega)\cap C_0(\overline{\Omega}),$ a solution of \eqref{MODELDEGQUASILINEAR}, such that $0<\underline{u}\leq u \leq \overline{U}$ in $\Omega$. 

\medskip
From \eqref{defFTheo2}, the function $u$ solves \eqref{DEGQUASILINEAR} and this completes the proof of the theorem.
\end{proof}

%\begin{proof}[Proof of Theorem \ref{theo3}.] The existence of a super solution $\overline{U}$ follows from Proposition \ref{SUBSPERSLNS0}. From (f3), we can fix any $r\in (1,p-1)$ and a nonlinearity $\tilde{f}:\R \to \R$ satisfying that $\tilde{f} \in C[0,\infty)$, $\tilde{f} \leq f$, $\tilde{f}(0)<0$ and for some $\mu_{\infty}\in [0, \infty)$

%\end{proof}

%%%%%%%%%%%%%%%%%%%%%%%%%%%%%%%%%%%%%%%%
%%%%%%%%%%%%%%%%%%%%%%%%%%%%%%%%%%%%%%%%
%                         APPENDIX
%%%%%%%%%%%%%%%%%%%%%%%%%%%%%%%%%%%%%%%%
%%%%%%%%%%%%%%%%%%%%%%%%%%%%%%%%%%%%%%%%

\end{document}